\documentclass[sn-mathphys-num]{sn-jnl}

\usepackage{graphicx}
\usepackage{multirow}
\usepackage{amsmath, amssymb, amsfonts}
\usepackage{amsthm}
\usepackage{mathrsfs}
\usepackage[title]{appendix}
\usepackage{xcolor}
\usepackage{textcomp}
\usepackage{manyfoot}
\usepackage{booktabs}
\usepackage{algorithm}
\usepackage{algorithmicx}
\usepackage{algpseudocode}
\usepackage{listings}
\usepackage{braket}
\usepackage{array}
\usepackage[caption=false]{subfig}
\usepackage{pgfplots}
\pgfplotsset{compat=1.18} 
\usepackage{amsfonts} 
\usepackage[utf8]{inputenc}
\usepackage[colorinlistoftodos]{todonotes}
\usepackage[T1, T2A]{fontenc}
\usepackage[english]{babel}
\usetikzlibrary{shapes.geometric}
\usepackage{rotating}
\usepackage{esvect}
\usepackage{float}
\usepackage{caption}
\usepackage{amssymb}
\usepackage{tikz}
\usepackage{tkz-base}
\usepackage{tkz-euclide}
\usetikzlibrary{calc}
\usepackage{relsize}
\usepackage{mathrsfs}
\usetikzlibrary{arrows}
\usepackage{soul}
\usepackage{hyperref}
\usepackage{wrapfig}
\usepackage{xurl}
\usepackage{enumitem}

\usetikzlibrary{fadings}

\definecolor{airforceblue}{rgb}{0.36, 0.54, 0.66}
\definecolor{antiquebrass}{rgb}{0.8, 0.58, 0.46}
\definecolor{applegreen}{rgb}{0.55, 0.71, 0.0}
\definecolor{amber(sae/ece)}{rgb}{1.0, 0.49, 0.0}
\definecolor{arsenic}{rgb}{0.23, 0.27, 0.29}
\definecolor{armygreen}{rgb}{0.29, 0.33, 0.13}
\definecolor{ashgray}{rgb}{0.7, 0.75, 0.71}
\definecolor{auburn}{rgb}{0.43, 0.21, 0.1}
\definecolor{aurometalsaurus}{rgb}{0.43, 0.5, 0.5}
\definecolor{babyblue}{rgb}{0.54, 0.81, 0.94}
\definecolor{azure(web)(azuremist)}{rgb}{0.94, 1.0, 1.0}
\definecolor{babyblue}{rgb}{0.54, 0.81, 0.94}
\definecolor{bole}{rgb}{0.47, 0.27, 0.23}
\definecolor{brightgreen}{rgb}{0.4, 1.0, 0.0}
\definecolor{britishracinggreen}{rgb}{0.0, 0.26, 0.15}
\definecolor{bulgarianrose}{rgb}{0.28, 0.02, 0.03}
\definecolor{cadmiumyellow}{rgb}{1.0, 0.96, 0.0}
\definecolor{carolinablue}{rgb}{0.6, 0.73, 0.89}
\definecolor{chartreuse(traditional)}{rgb}{0.87, 1.0, 0.0}
\definecolor{chartreuse(web)}{rgb}{0.5, 1.0, 0.0}
\definecolor{cinereous}{rgb}{0.6, 0.51, 0.48}
\definecolor{coolblack}{rgb}{0.0, 0.18, 0.39}
\definecolor{deepskyblue}{rgb}{0.0, 0.75, 1.0}
\definecolor{darkpastelblue}{rgb}{0.47, 0.62, 0.8}
\definecolor{ceil}{rgb}{0.57, 0.63, 0.81}
\definecolor{cadetgray}{rgb}{0.57, 0.64, 0.69}
\definecolor{bluedefrance}{rgb}{0.19, 0.55, 0.91}

\newcommand{\fixskip}{\medskip}
\long\def\notforarxiv#1\endnotforarxiv{}

\newcommand{\bluenew}[1]{
\begingroup #1\endgroup}

\newcommand*{\Scale}[2][4]{\scalebox{#1}{$#2$}}%

\newcommand{\hiddencomment}[1]{}

\renewcommand{\descriptionlabel}[1]{\hspace{\labelsep}\textit{#1}}

\theoremstyle{thmstyleone}%
\newtheorem{theorem}{Theorem}%
\newtheorem{corollary}{Corollary}%
\newtheorem{lemma}{Lemma}
\newtheorem{prop}{Proposition}%
\newtheorem{conjecture}{Conjecture}%

\theoremstyle{definition}
\newtheorem{example}{Example}%
\newtheorem{remark}{Remark}%
\newtheorem{problem}{Problem}%
\renewcommand{\descriptionlabel}[1]{\hspace{\labelsep}\textit{#1}}

\newtheorem{definition}{Definition}%

\begin{document}

\title[Flexible polyhedral 
nets in isotropic geometry]{Flexible polyhedral 
nets in isotropic geometry}


\author[1]{\fnm{Olimjoni} \sur{Pirahmad}}\email{pirahmad.olimjoni@kaust.edu.sa}
\equalcont{These authors contributed equally to this work.}


\author[1,2]{\fnm{Helmut} \sur{Pottmann}}\email{helmut.pottmann@gmail.com}
\equalcont{These authors contributed equally to this work.}

\author[1]{\fnm{Mikhail} \sur{Skopenkov}}\email{mikhail.skopenkov@gmail.com}
\equalcont{These authors contributed equally to this work.}

\affil[1]{\orgdiv{CEMSE}, \orgname{King Abdullah University of Science and Technology}, \orgaddress{
\city{Thuwal}, \postcode{23955-6900}, \country{Saudi Arabia}}}

\affil[2]{\orgdiv{Institute of Discrete Mathematics and Geometry}, \orgname{TU Wien}, \orgaddress{\street{Wiedner Hauptstr. 8-10},  \city{Wien}, \postcode{1040}, \country{Austria}}}



\abstract{We study flexible polyhedral nets in isotropic geometry. This geometry has a degenerate metric, but there is a natural notion of flexibility. We study infinitesimal and finite flexibility, and classify all finitely flexible polyhedral nets of arbitrary size. We show that there are just two classes, in contrast to Izmestiev's rather involved classification in Euclidean geometry, for size $3\times3$ only. Using these nets to initialize the optimization algorithms, we turn them into approximate Euclidean mechanisms. We also explore the smooth versions of these classes.
%
}








\maketitle





\section{Introduction}\label{sec-introduction}

The present paper has been motivated by recent work on
flexible polyhedral surfaces, in particular, quad nets with planar 
faces (called Q-nets) that act as mechanisms; see Figure~\ref{fig:SOM-exhib}. One considers the faces as rigid bodies
and the edges as hinges, and then they allow for continuous
flexion. Probably the most important recent progress is due to
Izmestiev  \cite{izmestiev2017classification}, who classified all flexible
$3 \times 3$ Q-nets and showed more than 20 classes in his classification. This is an essential task since, due to a 
result by Schief et al.~\cite{Schief2008}, an $m \times n$ 
Q-net is (two-sided) flexible if and only if all its $3 \times 3$ sub-nets are. Until 2020, the only
  known types of flexible
$m \times n$ Q-nets ($m$ or $n$ greater than $3$)
have been the classical special examples of Voss nets and T-nets \cite{sauer:1970}. A contribution
by He and Guest \cite{he2020rigid} changed that: these authors have been able to patch $3 \times 3$ nets together
to larger flexible Q-nets. This is still no method for approximation or design of such flexible
structures since the result is hardly predictable by the provided generation process. Very recently, Nawratil \cite{Nawratil2024}
introduced another class of flexible $m \times n$ Q-nets, however, without a geometric characterization that would enable design. Jiang et al.~\cite{quadmech-2024} base their design of flexible Q-nets on numerical optimization, the most difficult problem being initial guesses. 
The present paper contributes
to that part; see Figure~\ref{fig:flexible-eucl}.

\begin{figure}[bthp]
    \centering
    \includegraphics[scale=1]{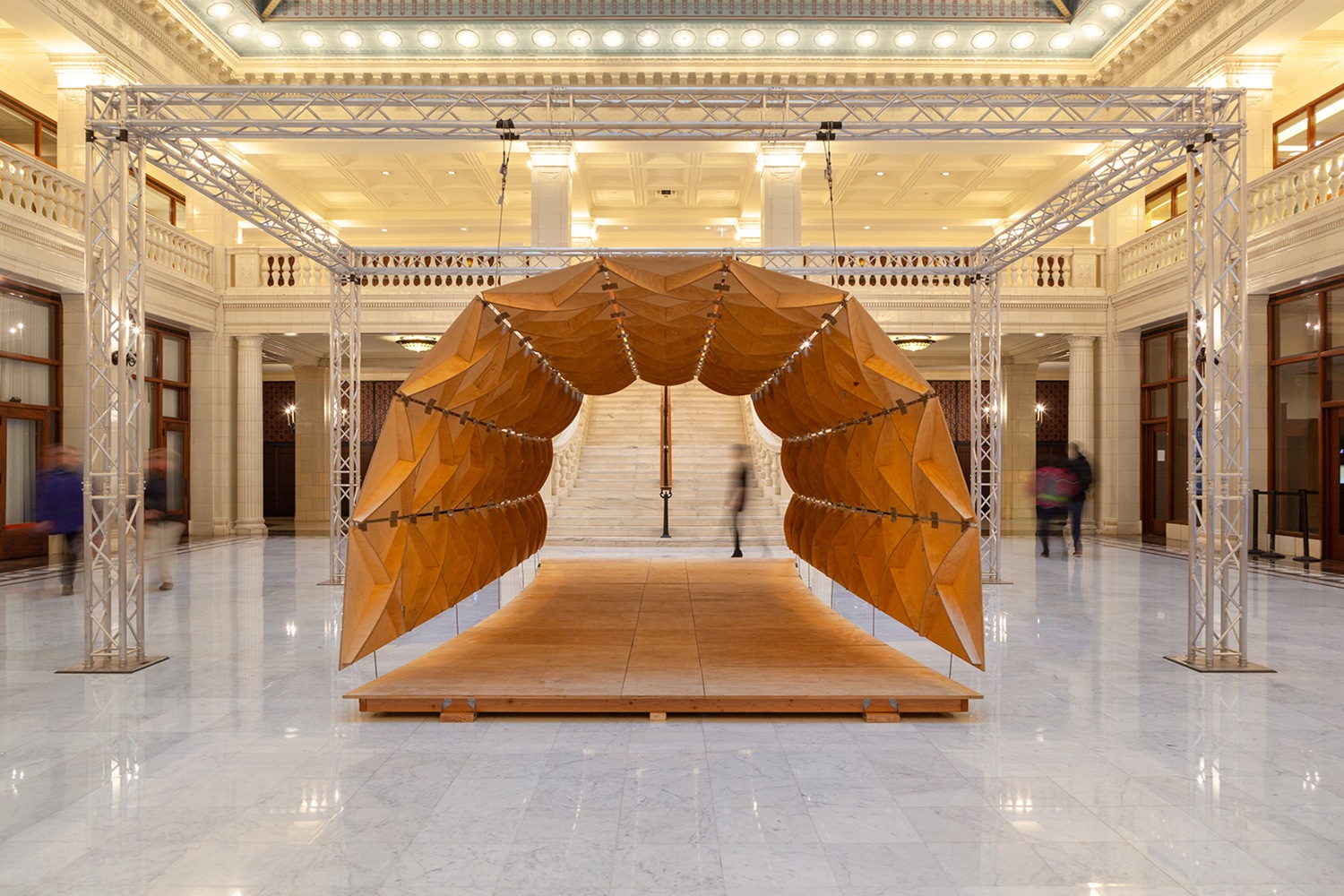}\quad\includegraphics[scale=0.12]{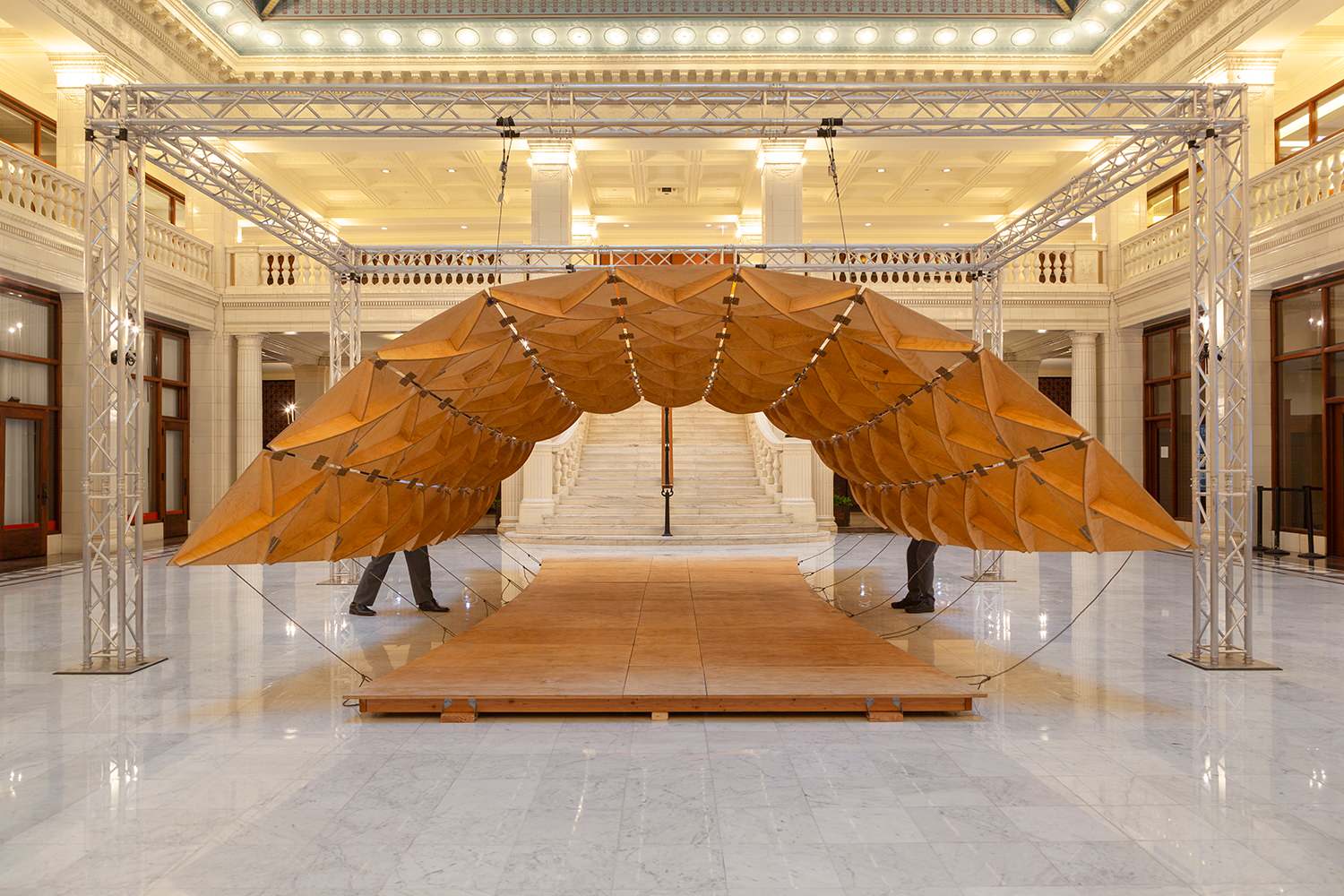}
    \caption{Kinematic sculpture in form of a Euclidean flexible Q-net of Voss type at the 2018 Chicago design week (\copyright\ Skidmore, Owings \& Merrill). 
}
    \label{fig:SOM-exhib}
\end{figure}

\begin{figure}[tbhp]
    \centering
    \includegraphics[scale=0.18]{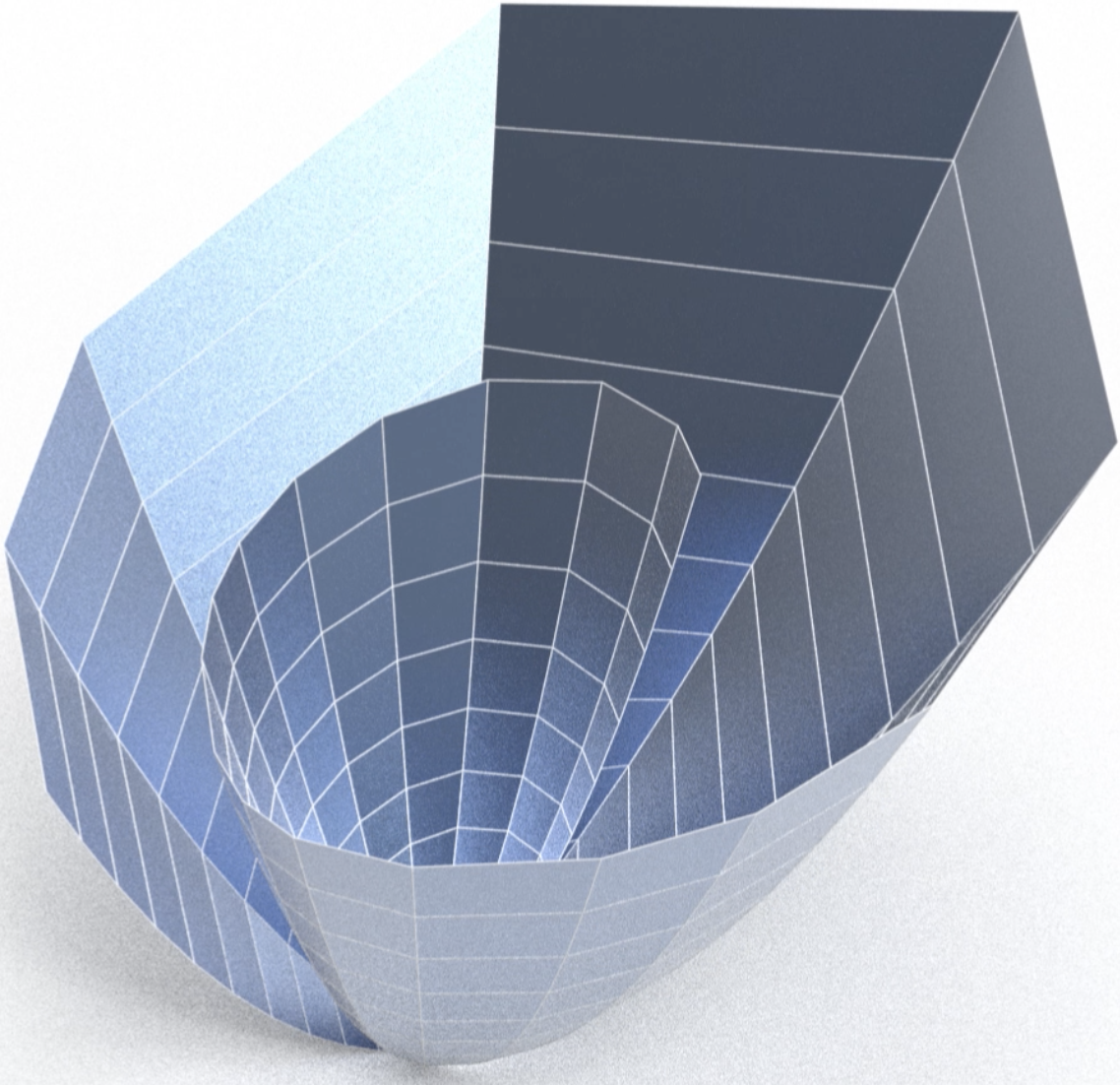}\qquad \includegraphics[scale=0.18]{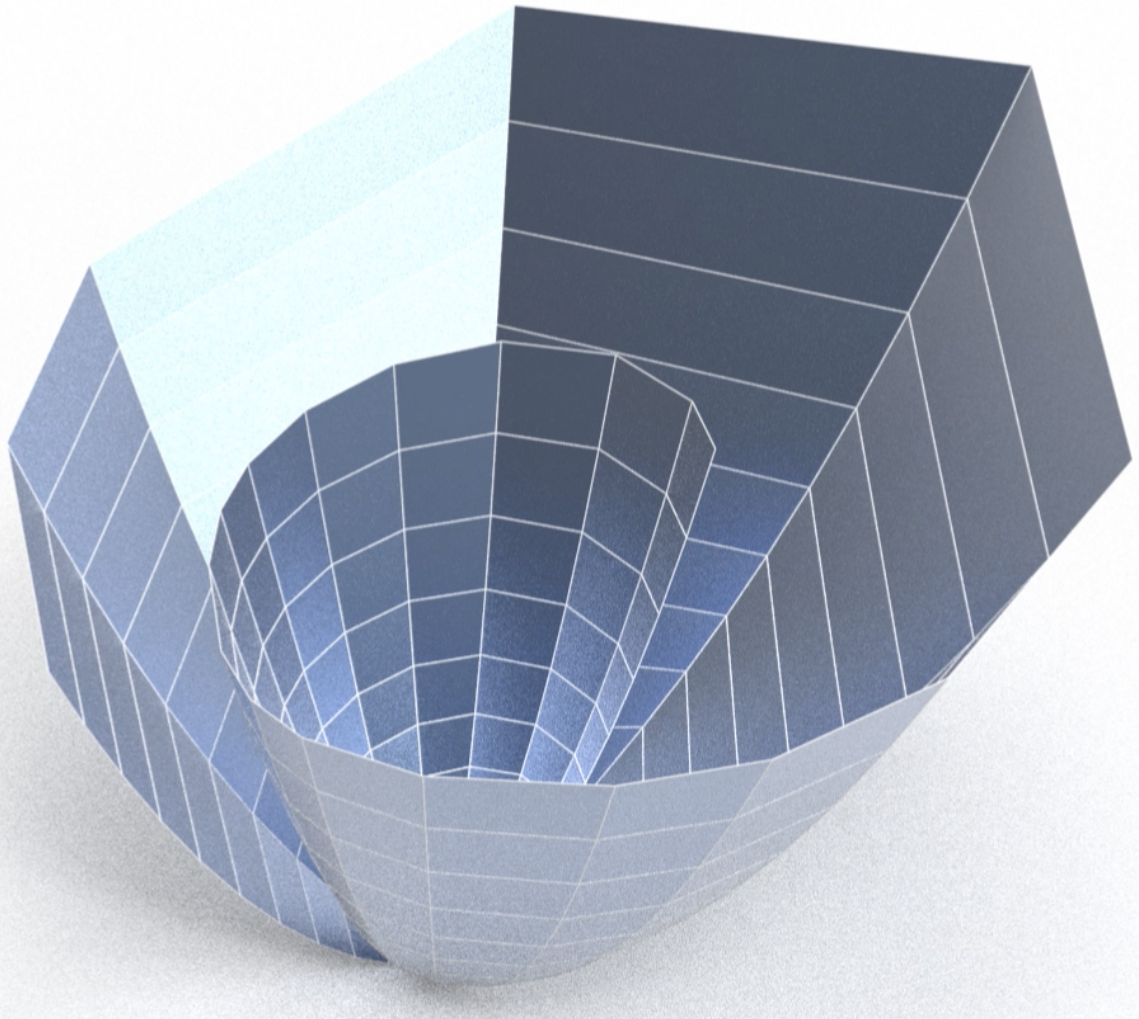}\qquad \includegraphics[scale=0.18]{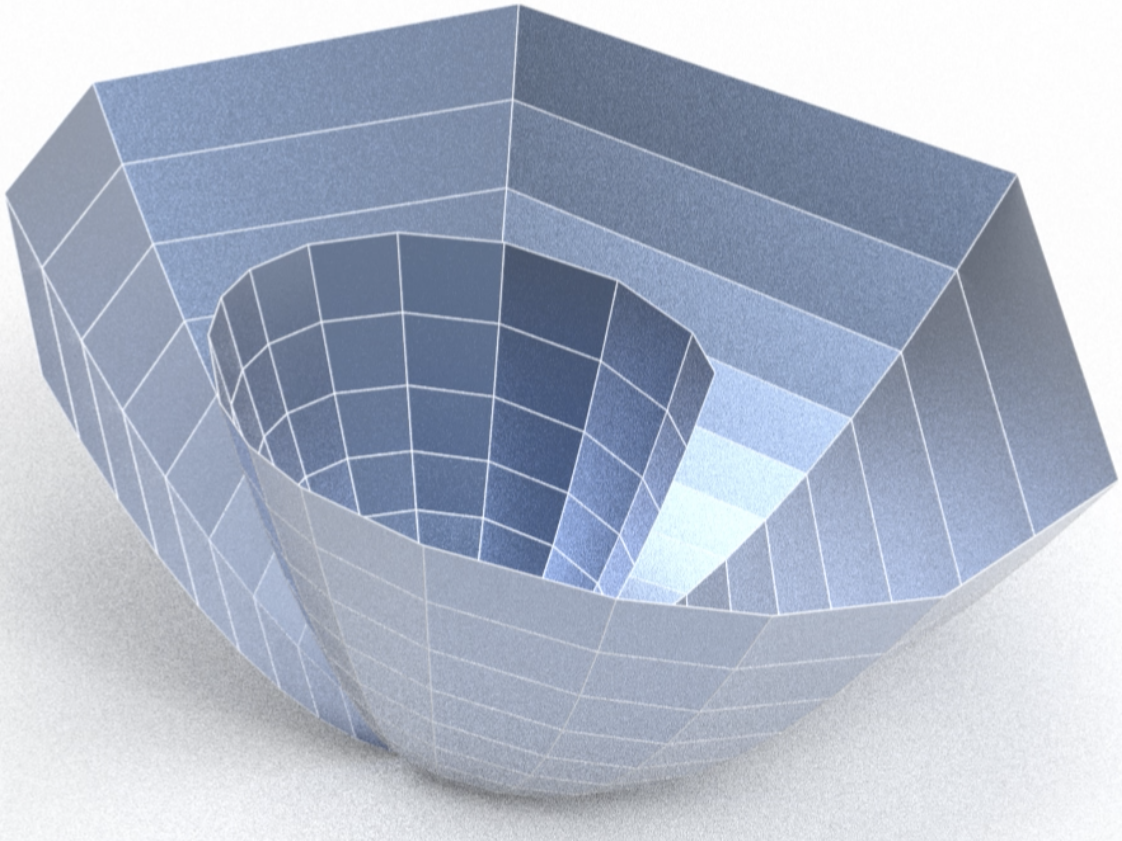}\\[0.2cm]
    \includegraphics[scale=0.2]{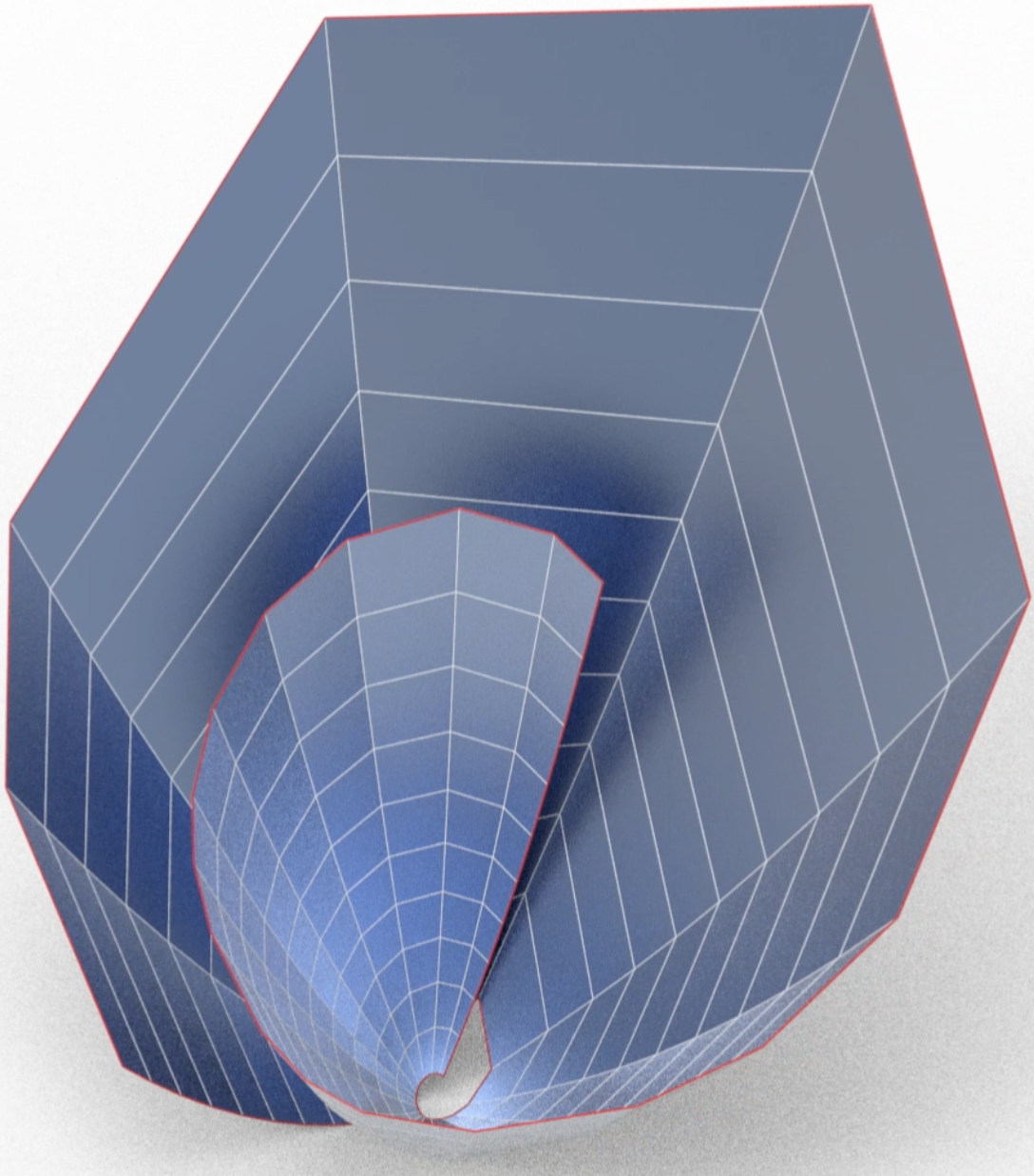}\qquad \includegraphics[scale=0.19]{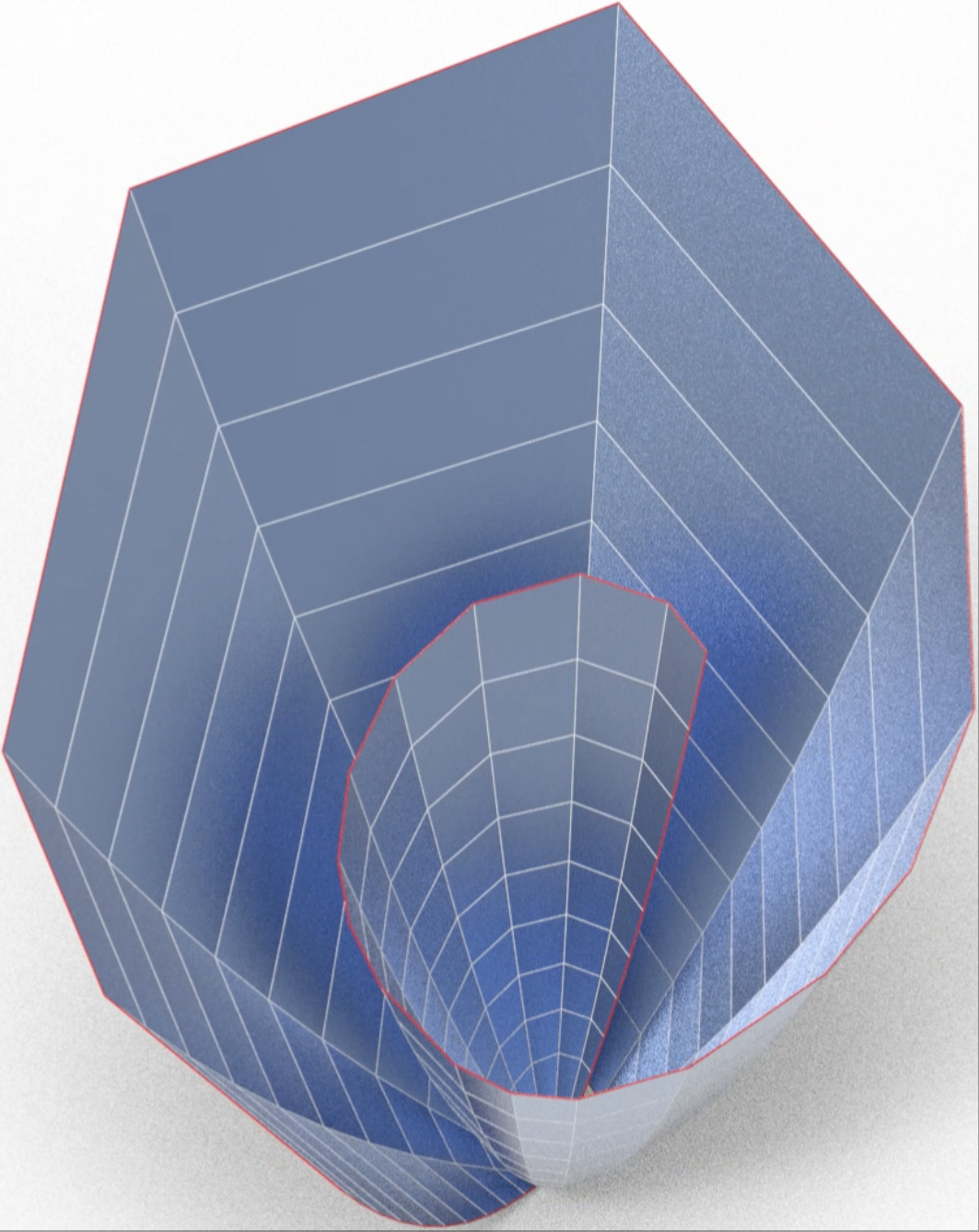}\qquad \includegraphics[scale=0.17]{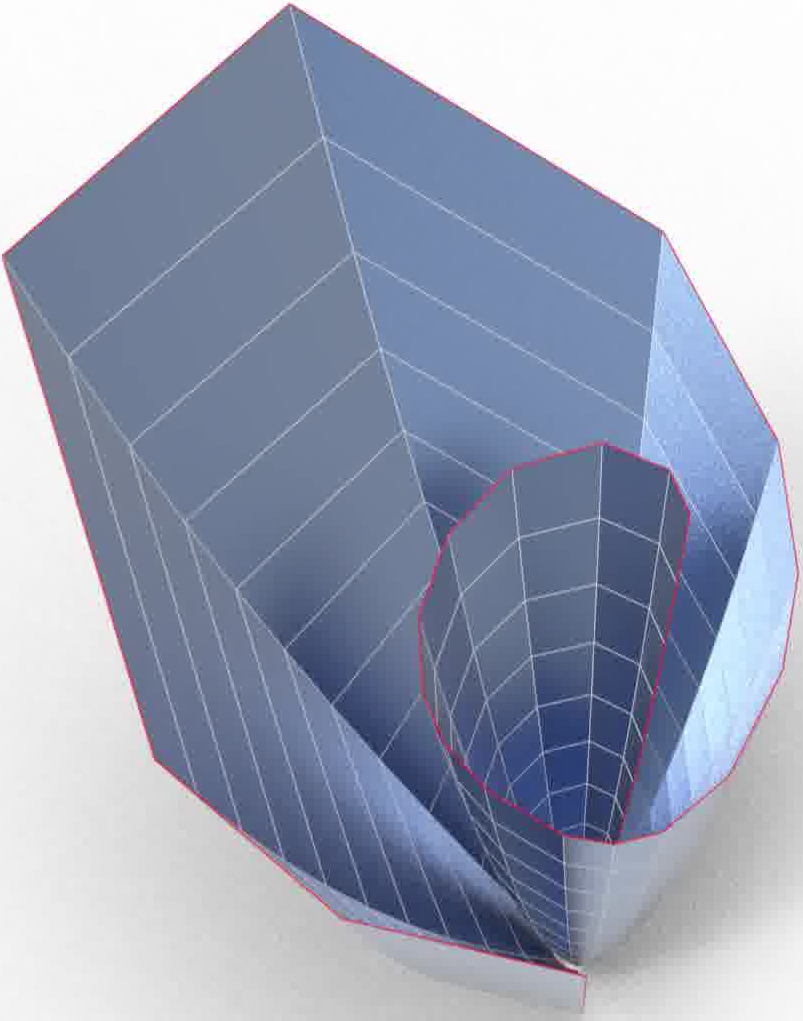}
    \caption{
    A few positions of an isotropic mechanism (upper row) and a Euclidean mechanism obtained by optimization (lower row) called ``flower''. The optimization leads to a very small change so that initial positions are very close. Images courtesy Caigui Jiang. }
    \label{fig:flexible-eucl}
\end{figure}

Flexible quad nets with not necessarily planar faces have also begun to emerge, as in the work by Nawratil \cite{Nawratil2022}
and Aikyn et al.~\cite{aikyn2024flexible}. 
The latter concerns a generalization of Izmestiev's orthodiagonal involutive type of Kokotsakis meshes, which holds promise to be
applicable to generalizations of other Izmestiev types as well.

To gain a better understanding, we therefore turned to a strategy
that has been proven effective in other problems, namely, to first study
the counterpart in the so-called \emph{isotropic geometry}.  This
simple non-Euclidean geometry has a degenerate metric based on
the semi-norm $\|(x,y,z)\|_i=\sqrt{x^2+y^2}$. One can come
up with a 6-dimensional group of isotropic rigid body motions 
and then develop a rich theory around that 
(see \cite{strubecker1,strubecker2,strubecker3,strubecker4,Sachs:1990,petrov-tikhomirov-06}). Isotropic geometry can be viewed as a structure-preserving simplification of Euclidean geometry and has 
applications 
within Euclidean geometry as well
(see, e.g., \cite{BMw-2024, Millar2022, pellis-smooth-2019, pottmann-2009-lms, Skopenkov-etal-12, Skopenkov-etal-20}).

However, until now, there has been no study of isometric deformations
of surfaces in isotropic 3-space $I^3$, apparently due to the degeneracy at first sight. Looking 
just at the metric is not enough. However, when adding the additional
requirement that isometric surfaces must have the same isotropic
Gaussian curvature at corresponding points, one can develop a theory
that does not degenerate. On the contrary, it provides additional
new insight into concepts and relations in statics and discrete
differential geometry \cite{isometric-isotropic}.

If we apply this recent concept of isometry, one can formulate 
a meaningful concept of flexible polyhedral nets in $I^3$. The problem becomes
even easier to deal with when applying the well-known metric
duality of isotropic 3-space. One arrives at \emph{area-preserving 
continuous Combescure transformations of Q-nets}~\cite{pirahmad2024area}. Two Q-nets are
\emph{Combescure transforms} of each other if they possess parallel corresponding
edges. This concept is even meaningful in the plane, and it belongs
to affine geometry. A Q-net that allows for a continuous family of area-preserving continuous Combescure transformations is briefly called \emph{deformable}.

Area preserving continuous Combescure transformations are of interest
without the relation to mechanisms in $I^3$ since Combescure transformations
are an essential concept in discrete differential geometry \cite{bobenko-2008-ddg}. Q-nets which allow for infinitesimal area-preserving Combescure
transformations turn out to be exactly the well-studied K{\oe}nigs nets,
and the velocity diagram of such an infinitesimal transformation is 
precisely the so-called Christoffel dual net. In 3-space, these nets appear as
relative minimal surfaces. Elaborating on the special cases that arise
when asking for continuous area-preserving Combescure transformations
is thus an interesting topic on its own.

Fortunately, in prior work~\cite{pirahmad2024area} we have
been able to come up with a complete classification of all deformable
nets in the above sense. The approach uses elementary geometry and
elementary methods of algebraic geometry. They fall into two classes. Class (i) is composed of two interleaved cone-cylinder nets, a particular case of cone-nets studied in \cite{kilian2023smooth}. Class (ii) is characterized by the existence of a Christoffel dual with equal areas of corresponding faces. This implies in general a visually non-smooth behavior. As a result of that, the smooth analogs of (ii) are just translational nets, and a special
case of the smooth analogs of type (i), which are smooth cone-cylinder nets, also known as scale-translational surfaces.

While the Euclidean classification of flexible Q-nets has so far only been done for nets with $3 \times 3$ faces and led to a large number of classes \cite{izmestiev2017classification},
the isotropic classification, which we give in the present work, is for arbitrary $m \times n$ nets and has only two classes: (i) and (ii). For class~(i), on each parameter line of one family, any two adjacent
non-boundary vertices are dual-affine-symmetric. 
Class (ii) includes nets with equal opposite ratios of neighboring non-boundary vertices. 

In addition to getting this classification, we characterize infinitesimal isotropic flexibility and study a particularly interesting sub-class of class (i), called generalized T-nets. Their smooth analogs 
are so-called generalized smooth T-nets. Generalized T-nets, both discrete and smooth, feature planar parameter lines, with one family of parameter lines lying in isotropic planes.

\textbf{Organization of the paper.} In Section~\ref{sec:final}, we provide an overview of the necessary concepts in isotropic geometry relevant to this paper and introduce the notion of dual-convex $m\times n$ nets. In Sections~\ref{sec-infinitesimal}--\ref{sec-smooth}, we study the infinitesimal flexibility of $m\times n$ nets, their finite flexibility, and the finite flexibility of smooth nets, respectively. These sections are independent, so the reader can directly proceed to the section of most interest (in a few cases, when one section uses a notion or a result from the other, we give a reference). Section~\ref{sec:concl-open-pr} presents the conclusion and proposes several open problems. 

\section{Preliminaries} \label{sec:final}


\subsection{Isotropic geometry} 
\label{ssec:isotropic}


Let $I^3$ be the $3$-dimensional real affine space (with a basis $e_1,e_2,e_3$) equipped with the \emph{isotropic semi-norm} $\|(x,y,z)\|_i:=\sqrt{x^2+y^2}$.
Vectors, lines, and planes parallel to the $z$-axis are called \emph{isotropic}. 
Isotropic vectors have vanishing isotropic semi-norm, thus we introduce the \emph{replacing semi-norm} 
$\|(0,0,z)\|_r:=|z|$ to distinguish them.
An affine transformation 
that preserves the isotropic semi-norm, the replacing semi-norm of isotropic vectors, and the orientation in the projection to the plane $z=0$ has the form
\begin{equation*}
  \mathbf{x} \mapsto A\cdot \mathbf{x}+\mathbf{b}, \quad 
A=  \begin{pmatrix} \cos\phi & -\sin\phi & 0  \\
 \sin\phi & \cos\phi & 0  \\
 c_1 & c_2 & 1 \end{pmatrix} 
\end{equation*}
for some values of the parameters $\mathbf{b}\in\mathbb{R}^3$ and $\phi,c_1,c_2\in\mathbb{R}$. Such transformations form the 6-parametric group $G^6$ of \emph{isotropic congruences}. 

These transformations appear as Euclidean congruences in the projection onto the plane $z=0$, which we call \emph{top view}. The top view of a point $P$ 
is denoted by $\overline P$. 
The \emph{isotropic distance} between two points and the \emph{isotropic angle} between two non-isotropic lines or vectors are expressed through the isotropic semi-norm in the usual way and appear in the top view as Euclidean distances and angles. 
Similarly, the \emph{oriented isotropic area} of a polygon is the oriented area of the top view. 

Points $P$ and $Q$ with the same top view are called \emph{parallel}. The isotropic distance $\|\overrightarrow{PQ}\|_i$ between them vanishes. One defines the \emph{replacing distance} between them to be the replacing norm $\|\overrightarrow{PQ}\|_r$.

The isotropic angle $\angle(p,q)$ between two non-isotropic 
planes $p,q$ is the absolute value of the difference of the slopes of the two planes, where the \emph{slope} is the tangent of the Euclidean angle between a given plane and the $xy$-plane. 
By definition, an isotropic plane (or line) is \emph{orthogonal} to any non-isotropic plane (or line).

In contrast to Euclidean geometry, isotropic geometry possesses a \emph{metric duality}. It is defined as the polarity with respect to the \emph{unit isotropic sphere} (paraboloid of revolution) $2z = x^2 + y^2$, which maps a point $P=(P^1,P^2,P^3)\in I^3$ to the plane $P^*$ with the equation $z = P^1x+P^2y-P^3$, and vice versa. 
The isotropic angle between two non-isotropic planes equals the isotropic distance between the points metric dual to them. 
The metric dual of a non-isotropic line $PQ$ is the intersection line of the non-isotropic planes $P^*$ and $Q^*$, and thus also a 
 non-isotropic line.

Isotropic lines and planes play a special role and are usually excluded as tangent spaces in differential geometry. A point of a surface is \emph{admissible} if the tangent plane at the point is non-isotropic, and a surface is \emph{admissible} if it has only admissible points. Hereafter by a \emph{surface} we mean the image of a proper injective $C^3$ map of a closed planar domain into $I^3$ with nondegenerate differential at each point.
An admissible surface can be locally represented as the graph of a function
$ z=z(x,y). $  The \emph{metric dual} of a surface is the set of points metric dual to the tangent planes of the surface.

The \emph{isotropic Gaussian map} takes a point of the surface to the point of the isotropic unit sphere such that the tangent planes at the two points are parallel. If this map is injective, then the oriented isotropic area of the image is called the \emph{total isotropic Gaussian curvature} of the surface. It equals the oriented isotropic area of the metric dual surface.
It also equals $\int K\, dxdy$, where $K:=z_{xx}z_{yy}-z_{xy}^2$ is the \emph{isotropic Gaussian curvature} of the surface $z=z(x,y)$.
This follows from the fact that $K$ measures the distortion in the local area of the isotropic Gauss map and that the top views of the isotropic Gauss image and the metric dual of a surface agree. 

\subsection{Dual-convex $m\times n$ nets}
\label{Subsection-dual-convex}

Now we discuss nets in $I^3$. To make the exposition more visual and less technical, we impose minor convexity restrictions. The first one is the convexity of faces, imposed in the definition of an $m\times n$ net below. The second one is equivalent to the convexity of faces of the metric dual; it is imposed in the definition of a dual-convex net.

By an $m\times n$ \emph{net} (\emph{with convex faces}) we mean a collection of $(m+1)(n+1)$ points $F_{ij}\in I^3$ indexed by two integers $0\le i\le m$ and $0\le j\le n$ (i.e., a map $\{0,\dots,m\}\times \{0,\dots,n\}\to I^3$), such that $F_{ij},F_{i+1,j},F_{i+1,j+1},F_{i,j+1}$ are consecutive vertices of a convex quadrilateral for all $0\le i< m$ and $0\le j< n$ (here $m,n\ge 1$). See Figure~\ref{figure:lemma-eqqq}.
The latter quadrilaterals are denoted by $p_{ij}$ and called \emph{faces}, and their sides are called \emph{edges}. 
Faces and edges are \emph{labeled}, 
i.e., equipped with an assignment of indices to their vertices (see~\cite[\S2]{Mor20} for a detailed discussion of labeled polygons). Points $F_{ij}$ are also called \emph{vertices} of the $m\times n$ net. 
A vertex $F_{ij}$ with $i\in\{0,m\}$ or $j\in\{0,n\}$ is a \emph{boundary} vertex. The faces $p_{i-1,j-1},p_{i,j-1},p_{ij},p_{i-1,j}$ are called the \emph{four consecutive faces around a non-boundary vertex $F_{ij}$}. See Figure~\ref{figure:lemma-eqqq} for $i=j=1$.
The plane of the face $p_{ij}$ is still denoted by $p_{ij}$. Two families of broken lines $F_{i0}\dots F_{in}$ for $0\le i\le m$ and $F_{0j}\dots F_{mj}$ for $0\le j\le n$ are called \emph{families of parameter lines}.

An $m\times n$ net has $mn$ faces, thus the~name.
In what follows, the phrase ``an $m\times n$ net consisting of points $F_{ij}$ indexed by two integers $0\le i\le m$ and $0\le j\le n$ 
and having the faces $p_{kl}$ indexed by $0\le k< m$ and $0\le l< n$'' is abbreviated to ``an $m\times n$ net $F_{ij}$'' or ``an $m\times n$ net with faces $p_{kl}$''.

\begin{figure}[htbp]
  \centering
\begin{tikzpicture}[scale=0.12]

\coordinate (A1) at (11.7, -16);
\coordinate (O) at (8.2, -24.8);
\coordinate (A2) at (25, -21.9);
\coordinate (B2) at (22.1, -36.3);
\coordinate (A3) at (12.3, -34.3);
\coordinate (B3) at (-5.3, -37.3);
\coordinate (A4) at (-5.6, -23.2);
\coordinate (B4) at (0.1, -11.7);
\coordinate (B1) at (21.5, -11.8);

\fill[babyblue!5] (O) -- (A1) -- (B1) -- (A2) -- cycle;
\fill[antiquebrass!8] (O) -- (A2) -- (B2) -- (A3) -- cycle;
\fill[britishracinggreen!5] (O) -- (A3) -- (B3) -- (A4) -- cycle;
\fill[lime!8] (O) -- (A4) -- (B4) -- (A1) -- cycle;

\draw[black, thin] (A1) -- (O) -- (A2) -- (B2) -- (A3) -- (B3) -- (A4) -- (B4) -- (A1) -- (B1) -- (A2);
\draw[black, thin] (A4) -- (O) -- (A3);

\filldraw[black] (O) circle (4pt);
\filldraw[black] (9.2, -24.3) circle (0pt) node[anchor=north east]{$F_{11}$};

\filldraw[black] (25, -21.9) circle (4pt) node[anchor=west]{$F_{21}$};
\filldraw[black] (12.3, -34.3) circle (4pt);
\filldraw[black] (11.4, -35.1) circle (0pt) node[anchor=north]{$F_{10}$};
\filldraw[black] (-5.6, -23.2) circle (4pt) node[anchor=east]{$F_{01}$};
\filldraw[black] (11.7, -16) circle (4pt); 
\filldraw[black] (9.5, -15.5) circle (0pt) node[anchor=south west]{$F_{12}$};

\filldraw[black] (21.5, -11.8) circle (4pt) node[anchor=west]{$F_{22}$};
\filldraw[black] (22.1, -36.3) circle (4pt) node[anchor=west]{$F_{20}$};
\filldraw[black] (-5.3, -37.3) circle (4pt) node[anchor=east]{$F_{00}$};
\filldraw[black] (0.1, -11.7) circle (4pt) node[anchor=east]{$F_{02}$};

\filldraw[black] (17.63456,-18.97232) node[anchor=center]{$p_{11}$};
\filldraw[black] (19.8,-28.64956) node[anchor=east]{$p_{10}$};
\filldraw[black] (2.5,-30.5) node[anchor=center]{$p_{00}$};
\filldraw[black] (1.5,-18.94961) node[anchor=west]{$p_{01}$};
\end{tikzpicture}

\caption{A $2\times 2$ net. See Subsection~\ref{Subsection-dual-convex}.}
\label{figure:lemma-eqqq}
\end{figure}
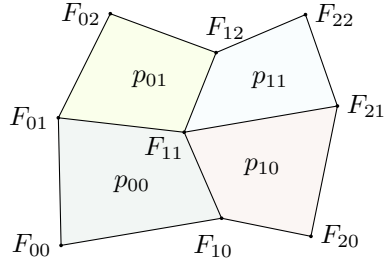


Edges or faces spanned by points with the same indices in two $m\times n$ nets are called \emph{corresponding}. 
Two $m\times n$ nets are \emph{parallel} or \emph{Combescure transformations} of each other if their corresponding edges are parallel. 
Two $m\times n$ nets are \emph{v-parallel} if their vertices with the same indices are parallel. 

Given a convex quadrilateral in a plane in $I^3$, and a point $O$ not in the plane, the union of all rays emanating from $O$ and intersecting the quadrilateral is called a \emph{convex $4$-hedral angle with the vertex $O$}. The union of all rays emanating from $O$ and intersecting one side of the quadrilateral is a \emph{flat angle} of the $4$-hedral angle.

To state our results, we need the following new notion. 

\fixskip
\begin{definition} 
    \label{admissible-conv-pol--angle}
(See Figure~\ref{figure:admissible-angle} to the left and middle.)
A convex $4$-hedral angle in $I^3$ is \emph{admissible} if the isotropic line through its vertex intersects its interior. 
An $m\times n$ net is \emph{dual-convex} if $m,n\ge 2$ and the planes of the four consecutive faces around each non-boundary vertex are the planes of the four consecutive flat angles of some admissible $4$-hedral angle. 
\end{definition}
\fixskip

Notice that any $m\times n$ net such that the union of faces is a graph of a convex or concave function $z(x,y)$ is necessarily a dual-convex net. However, the converse is not true: the union of faces of a dual-convex $m\times n$ net is not necessarily a graph of a convex function. In Definition~\ref{admissible-conv-pol--angle}, the corresponding face angles at a vertex of a dual-convex net and the admissible \(4\)-hedral angle may differ;
namely, the face angles can be vertical or supplementary. 
For instance, one can take four vectors generating an admissible $4$-hedral angle (Figure~\ref{figure:admissible-angle} to the left) and reverse the direction of one vector (Figure~\ref{figure:admissible-angle} to the middle). Convex quadrilaterals constructed on the resulting vectors 
do not form a convex  $4$-hedral angle anymore but they still lie in the planes of flat angles of the initial convex $4$-hedral angle. So,
a dual-convex net can have ``negatively curved'' vertices (where the neighboring faces form a graph of neither convex nor concave function) and ``non-admissible'' ones (where they do not form a graph of a function or even intersect each other at interior points). 
However, a dual-convex net can have neither 
faces lying in isotropic planes 
nor adjacent faces lying in the same plane. 


The \emph{metric dual} of a dual-convex $m\times n$ net with faces $p_{kl}$ is the $(m-1)\times(n-1)$ net formed by the points 
$p_{kl}^*$ metric dual to the planes of the faces $p_{kl}$, where $0\leq k< m,0\leq l< n$. 
The metric dual constructed in this way has convex faces if and only if the initial $m\times n$ net is dual-convex  (by Lemma~\ref{l-dual-convex} below), hence the terminology. This also explains the condition $m,n\ge 2$ in Definition~\ref{admissible-conv-pol--angle}.




\subsection{Discrete isotropic Gaussian curvature}
\label{ssec-curvature}

The famous Theorema Egregium by Gauss says that the Gaussian curvature is invariant under smooth isometries. In Euclidean geometry, there is a well-known discrete analog of Gaussian curvature with similar properties: the (total) Gaussian curvature at a point of an $m\times n$ net is $2\pi$ minus the sum of the face angles meeting at the point. In isotropic geometry, the analogous expression vanishes identically (for ``admissible'' vertices). Similarly to other cases when a straightforward analog of a Euclidean invariant degenerates, we look for a \emph{replacing} invariant that is non-degenerate. Thus, instead of the sum of flat angles around a point, we look at the measure of the dual solid angle as follows.

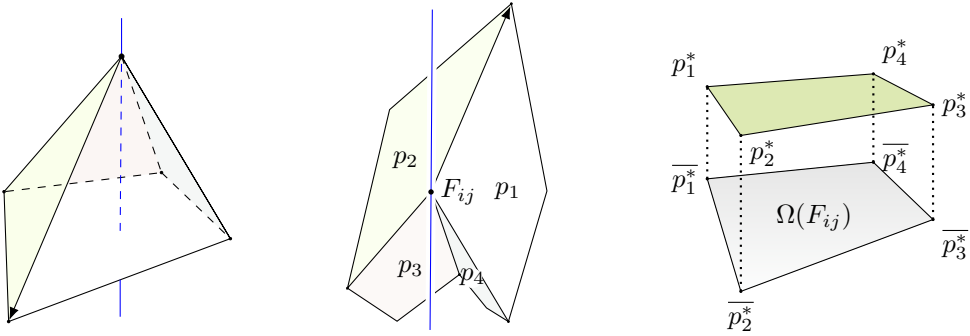
\begin{figure}[htbp]
  \centering
  \begin{tikzpicture}[scale=0.45]

\coordinate (A) at (-3.66577, 2.17019);
\coordinate (B) at (-7.10198, -1.81635);
\coordinate (C) at (-6.96776, -5.64182);
\coordinate (D) at (-0.4846, -3.19889);
\coordinate (E) at (-2.498, -1.2526);
\coordinate (H) at (-3.66577, 3.29769);
\coordinate (F) at (-3.70604, -2.9707);
\coordinate (G) at (-3.70604, -4.40693);
\coordinate (I) at (-3.70604, -5.5);

\fill[britishracinggreen!4] (A) -- (E) -- (D) -- cycle;
\fill[antiquebrass!8] (A) -- (B) -- (E) -- cycle;
\fill[lime!8] (A) -- (B) -- (C) -- cycle;

\draw[black, thin] (D) -- (A) -- (B) -- (C) -- (D) -- (A);
\draw[blue, thin] (A) -- (H);
\draw[blue, thin] (G) -- (I);
\draw [-Latex] (A) -- (-6.95, -5.6);

\draw[black, dashed] (A) -- (E) -- (B);
\draw[blue, dashed] (F) -- (A);
\draw[black, dashed] (E) -- (D);

\filldraw[black] (A) circle (2pt) node[anchor=north]{};
\filldraw[black] (B) circle (1pt) node[anchor=north]{};
\filldraw[black] (C) circle (1pt) node[anchor=north]{};
\filldraw[black] (D) circle (1pt) node[anchor=north]{};
\filldraw[black] (E) circle (1pt) node[anchor=north]{};

\end{tikzpicture}\qquad\qquad
\begin{tikzpicture}[scale=0.32]

\coordinate (A) at (-3.66577, 2.17019);
\coordinate (B) at (-7.10198, -1.81635);
\coordinate (C) at (-0.36379, 9.96877);
\coordinate (D) at (-0.4846, -3.19889);
\coordinate (E) at (-2.498, -1.2526);
\coordinate (H) at (-3.61208, 9.72716);
\coordinate (I) at (-3.70604, -3.54788);
\coordinate (F) at (-3.70604, -2.9707);
\coordinate (G) at (-3.70604, -4.40693);
\coordinate (T) at (-5.37042, 5.57955);
\coordinate (R) at (1.1, 2.21045);
\coordinate (U) at (-5.06173, -3.18547);
\coordinate (V) at (-1.38392, -2.63514);

\fill[lime!6] (A) -- (B) -- (T) -- (C) -- cycle;
\fill[britishracinggreen!4] (A) -- (D) -- (V) -- (E) -- cycle;
\fill[antiquebrass!6] (A) -- (E) -- (U) -- (B) -- cycle;

\draw[black, thin] (A) -- (D) -- (R) -- (C);
\draw[black, thin] (A) -- (B) -- (T) -- (C);
\draw[black, thin] (A) -- (E) -- (U) -- (B);
\draw[black, thin] (E) -- (V) -- (D);
\draw [-Latex] (A) -- (-0.4, 9.86877);
\draw[white, line width=3.5pt] (H) -- (I);
\draw[blue, thin] (H) -- (I);

\filldraw[black] (A) circle (2.5pt) node[anchor=west]{$F_{ij}$};
\filldraw[black] (B) circle (1pt) node[anchor=west]{};
\filldraw[black] (C) circle (1pt) node[anchor=north]{};
\filldraw[black] (D) circle (1pt) node[anchor=north]{};
\filldraw[black] (E) circle (1pt) node[anchor=east]{};

\draw[black] (-0.5, 2.21045) node[]{$p_1$};
\draw[black] (-4.5, -1) node[]{$p_3$};
\draw[black] (-4.7, 3.5) node[]{$p_2$};
\filldraw[white] (-1.97, -1.4) circle (11pt);
\draw[black] (-1.95, -1.36) node[]{$p_4$};
\end{tikzpicture}
\qquad\qquad
\begin{tikzpicture}[scale=0.39]

\coordinate (A) at (-3.66577, 2.17019);
\coordinate (B) at (-8.10198, -2.31635);
\coordinate (C) at (-6.96776, -6.14182);
\coordinate (D) at (-0.4846, -3.69889);
\coordinate (E) at (-2.498, -1.7526);
\coordinate (F) at (-8.10198, 0.81635);
\coordinate (G) at (-6.96776, -0.84182);
\coordinate (H) at (-0.4846, 0.19889);
\coordinate (I) at (-2.498, 1.2526);

\fill[shade] (B) -- (C) -- (D) -- (E) -- cycle;
\shade[top color=black!10, bottom color=britishracinggreen!1] (B) -- (C) -- (D) -- (E) -- cycle;
\draw[black, thin] (B) -- (C) -- (D) -- (E) -- (B);
\fill[applegreen!30] (F) -- (G) -- (H) -- (I) -- cycle;
\draw[black, thin] (F) -- (G) -- (H) -- (I) -- cycle;
\draw[black, dotted, thick] (B) --  (F);
\draw[black, dotted, thick] (C) --  (G);
\draw[black, dotted, thick] (D) --  (H);
\draw[black, dotted, thick] (E) --  (I);

\filldraw[black] (B) circle (1pt) node[anchor=east]{$\overline{p_1^*}$};
\filldraw[black] (C) circle (1pt) node[anchor=north]{$\overline{p_2^*}$};
\filldraw[black] (D) circle (1pt) node[anchor=north west]{$\overline{p_3^*}$};
\filldraw[black] (E) circle (1pt);
\filldraw[black] (-2.5, -1.6) circle (0pt) node[anchor=west]{$\overline{p_4^*}$};
\filldraw[black] (F) circle (1pt) node[anchor=south east]{${p_1^*}$};
\filldraw[black] (G) circle (1pt);
\filldraw[black] (-6.96776, -0.7) circle (0pt) node[anchor=north west]{${p_2^*}$};
\filldraw[black] (H) circle (1pt) node[anchor=west]{${p_3^*}$};
\filldraw[black] (I) circle (1pt) node[anchor=south west]{${p_4^*}$};

\draw[black] (-4.5, -3.6) node[]{$\Omega(F_{ij})$};


\end{tikzpicture}
\caption{Left: An admissible 4-hedral angle. The isotropic line (blue) through the vertex intersects the angle interior. Middle: Four consecutive faces $p_1,{p_2},{p_3},{p_4}$ around a vertex $F_{ij}$ of a dual-convex net. In the left and middle figures,
the vertex is the same (black point), the four emanating edges are the same except for the one equipped with an arrow (whose direction is reversed),  
and faces of the same color lie in the same plane. The union of the faces 
$p_1,{p_2},{p_3},{p_4}$ is not a graph of a function $z=z(x,y)$. 
Right: The metric dual of the sub-net formed by $p_1,{p_2},{p_3},{p_4}$ is a convex quadrilateral ${p_1^*}{p_2^*}{p_3^*}{p_4^*}$ (see Lemma~\ref{l-dual-convex}). 
The total isotropic Gaussian curvature $\Omega(F_{ij})$ concentrated at the vertex $F_{ij}$ is the oriented area of its top view $\overline{p_1^*}\,\overline{p_2^*}\,\overline{p_3^*}\,\overline{p_4^*}$. See Definitions~\ref{admissible-conv-pol--angle} and~\ref{def-curvature}.}
\label{figure:admissible-angle}
\end{figure}

\fixskip
\begin{definition} \label{def-curvature}
(See Figure~\ref{figure:admissible-angle} to the middle and right.)
Let $F_{ij}$ be a dual-convex $m\times n$ net. Consider 
the consecutive faces $p_1,p_2,p_3,p_4$ around a particular non-boundary vertex $F_{ij}$. The top views of the points metric dual to the planes of these faces form a convex quadrilateral 
$\overline{p_1^*}\,\overline{p_2^*}\,\overline{p_3^*}\,\overline{p_4^*}$ called the \emph{top view of the isotropic Gaussian image of $F_{ij}$}.
The oriented area of this quadrilateral is called the \emph{total isotropic Gaussian curvature concentrated at the vertex} $F_{ij}$, or the \emph{value of the discrete curvature form at} $F_{ij}$,
or simply the \emph{curvature at} $F_{ij}$:
\begin{equation}\label{eq-def-curvature}
\Omega(F_{ij}):= \mathrm{Area}\left(\overline{p_1^*}\,\overline{p_2^*}\,\overline{p_3^*}\,\overline{p_4^*}\right)
=\frac{1}{2} \sum_{j=1}^{4} \det\left(\overline{p_j^*}, \overline{p_{j+1}^*}\right),
\qquad\text{where $\overline{p_{5}^*}:=\overline{p_{1}^*}$}.
\end{equation}
(The vanishing $z$-coordinate of $\overline{p_j^*}$ is dropped under the determinant.)
\end{definition}
\fixskip

With a slight abuse of notation, the curvature at 
$F_{ij}$ 
is denoted by $\Omega(F_{ij})$ although it depends on the neighboring vertices as well.

This definition differs from the known analogs \cite
{pottmann2007discrete,bobenko-2010-curv} in the following aspects. 
Value~\eqref{eq-def-curvature} should be viewed as the \emph{total} Gaussian curvature concentrated at a point rather than the value of a function, discretizing the Gaussian curvature, at a particular point. In other words, \eqref{eq-def-curvature} should be viewed as a discrete analog of the \emph{curvature form} $\Omega=K\,dxdy$ rather than the Gaussian curvature $K$ itself. Thus $\Omega(F_{ij})$ is defined as a single area 
rather than the ratio of two areas as in \cite[Eq.~(22)]{pottmann2007discrete}. This allows us to bypass the construction of a discrete Gauss image and the approximation of the unit sphere by parallel meshes, which is rather tricky for square-grid combinatorics. In our setup, just the \emph{top view} of the isotropic Gauss image is sufficient, being the same as the top view of the metric dual.


\section{Infinitesimal flexibility}
\label{sec-infinitesimal}

\subsection{Definition} 

An \emph{infinitesimal isotropic congruence} is a vector field on $I^3$ of the form
\begin{equation}\label{eq-isotropic-congruence}
  V(\mathbf{x})= a\cdot \mathbf{x}+\mathbf{b}, \quad 
a=  \begin{pmatrix} 0 & -\phi & 0  \\
 \phi & 0 & 0  \\
 c_1 & c_2 & 0 \end{pmatrix} 
\end{equation}
for some 
$\mathbf{b}\in\mathbb{R}^3$ and $\phi,c_1,c_2\in\mathbb{R}$. Clearly, such vector fields form the Lie algebra of $G^6$. (We do not use and do not discuss the product in this Lie algebra.) 

Let us define the infinitesimal flexibility in $I^3$. It requires the usual \emph{face condition} meaning the infinitesimal rigidity of faces. Compared to Euclidean geometry, we impose an additional \emph{vertex condition}, which is automatic in Euclidean geometry, but becomes crucial in isotropic geometry.

\fixskip
\begin{definition} \label{def-infinitesimal-flexibility}
(Cf.~\cite[Definition~57]{isometric-isotropic}.)
An \emph{infinitesimal isotropic isometric deformation} of a dual-convex $m\times n$ net $F_{ij}$
is an indexed collection of vectors $V_{ij}\in\mathbb{R}^3$, 
where $0\leq i\leq m,0\leq j\leq n$, satisfying 
\fixskip
\begin{description}
    \item[\textbf{face condition:}] for each $0\leq k< m,0\leq l< n$ there is an infinitesimal isotropic congruence $V$ (depending on $k$ and $l$) such that $V_{ij}=V(F_{ij})$ for each $i\in\{k,k+1\}$ and $j\in\{l,l+1\}$; and   
    \item[\textbf{vertex condition:}] for each $0< i< m,0< j< n$ we have $\frac{d}{dt}\left.\Omega(F_{ij}+tV_{ij})\right|_{t=0}=0$.
\end{description}
\fixskip
(Here $F_{ij}+tV_{ij}$ is a dual-convex $m\times n$ net for all sufficiently
small $t$ by Lemma~\ref{l-affine} below, hence $\Omega(F_{ij}+tV_{ij})$ is well-defined.)

An infinitesimal isotropic isometric deformation is \emph{trivial} if there is an infinitesimal isotropic congruence $V$ such that $V_{ij}=V(F_{ij})$ for 
all $0\leq i\leq m,0\leq j\leq n$. 
A dual-convex $m\times n$ net 
is \emph{infinitesimally flexible} in $I^3$ if it has a nontrivial infinitesimal isotropic isometric deformation. 
\end{definition}
\fixskip

\subsection{Statement of the classification}
\label{ssec-infinitesimal-statements}

Let us characterize infinitesimally flexible $m\times n$ nets in $I^3$; those turn out to be essentially the same 
ones as in Euclidean geometry. In particular, by studying the former, we get new insight into the latter. The characterization below is stated in several equivalent ways; that is why the used known notions are defined after the statement so that the reader can choose the way that she or he likes most. 

\fixskip
\begin{prop} \label{prop-infinitesimal-characterization}
For a dual-convex $m\times n$ net, 
the following conditions are equivalent:
\begin{itemize}
    \item[\textup{\textbf{(I)}}] the $m\times n$ net is infinitesimally flexible in $I^3$;
    \item[\textup{\textbf{(D)}}] the metric dual $(m-1)\times (n-1)$ net is infinitesimally deformable;
    \item[\textup{\textbf{(K)}}] the metric dual $(m-1)\times (n-1)$ net is a K\oe nigs net;
    \item[\textup{\textbf{(V)}}] the $m\times n$ net has a 
    non-planar v-parallel $m\times n$ net with vanishing mixed curvature of the corresponding vertices;
    \item[\textup{\textbf{(R)}}] the $m\times n$ net has a reciprocal-parallel net; 
    \item[\textup{\textbf{(E)}}] the $m\times n$ net is infinitesimally flexible in Euclidean geometry.
\end{itemize}
\end{prop}
\fixskip

This result was essentially obtained 
in~\cite[Proposition~60 and Lemma~62]{isometric-isotropic} but our proof in \S\ref{ssec-infinitesimal-proofs} is slightly different. Instead of vectors $V_{ij}$ isotropic orthogonal to $F_{ij}$,
we use the isotropic vectors. 
This way we get a different construction of the reciprocal-parallel net; see~Lemma~\ref{l-deformable-reciprocal} below and cf.~\cite[Lemma~62]{isometric-isotropic}.

Let us define the known notions used in the proposition; cf.~\cite{isometric-isotropic}, \cite{bobenko-2008-ddg}, \cite{bobenko-2010-curv}, \cite{sauer:1970}, and \cite{Schief2008}, respectively:

\textbf{(D)} \label{def-(D)} An \emph{infinitesimal area preserving Combescure transformation}, or \emph{infinitesimal deformation}, of an $(m-1)\times (n-1)$ net $P_{ij}$ in $\mathbb{R}^d$ is a collection of vectors $C_{ij}\in\mathbb{R}^d$, where $0\leq i< m,0\leq j< n$,
such that the nets $P_{ij}(t):=P_{ij}+tC_{ij}$ are Combescure transformations of the net $P_{ij}$ for all sufficiently small $t\in\mathbb{R}$ and 
$$\frac{d}{dt}\left.\mathrm{Area}(P_{ij}(t)P_{i-1,j}(t)P_{i-1,j-1}(t)P_{i,j-1}(t))\right|_{t=0}=0.$$
A net $P_{ij}$ is \emph{infinitesimally deformable} if it has an infinitesimal deformation with not all vectors $C_{ij}$ equal. An infinitesimal deformation consisting of equal vectors $C_{ij}$ is referred to as \emph{trivial}.

\textbf{(K)} Two (labeled) quadrilaterals 
$ABCD$ and $A'B'C'D'$ are \emph{dual} if their corresponding sides and non-corresponding diagonals are parallel (see Figure~\ref{figure:dual-of-2-quads}):
$$
AB\parallel A'B', \quad BC\parallel B'C', \quad CD\parallel C'D', \quad DA\parallel D'A', \quad AC\parallel B'D', \quad BD\parallel C'D'.
$$
Two $(m-1)\times (n-1)$ nets are \emph{Christoffel dual} if their corresponding faces are dual. A \emph{K\oe nigs net} is a one having a Christoffel dual. K\oe nigs nets have many geometric characterizations; see~\cite[\S2.2.2]{pirahmad2024area}.

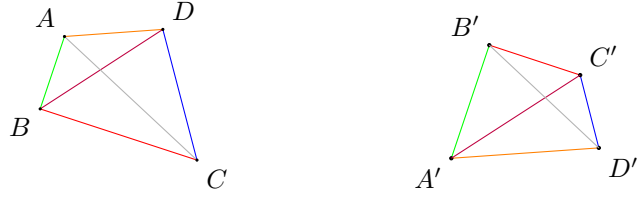
\begin{figure}[!t]
  \centering
\begin{tikzpicture}[scale=0.032]

\coordinate (A) at (54.48536, -61.2148);
\coordinate (B) at (40.45495, -7.05598);
\coordinate (C) at (0, -10);
\coordinate (D) at (-10, -40);

\draw[blue, thin] (A) -- (B);
\draw[orange, thin] (B) -- (C);
\draw[green, thin] (C) -- (D);
\draw[red, thin] (D) -- (A);

\draw[purple, name path=line 1] (B) -- (D);
\draw[gray!60, name path=line 2] (A) -- (C);

\filldraw[black] (A) circle (12pt) node[anchor=north west]{$C$};
\filldraw[black] (B) circle (12pt) node[anchor=south west]{$D$};
\filldraw[black] (C) circle (12pt) node[anchor=south east]{$A$};
\filldraw[black] (D) circle (12pt) node[anchor=north east]{$B$};

\end{tikzpicture}
\qquad\qquad\qquad
\begin{tikzpicture}[scale=0.05]

\coordinate (C) at (0, -10);
\coordinate (D) at (-10, -40);
\coordinate (B') at (23.91368, -17.93708);
\coordinate (A') at (28.81781, -37.24708);

\draw[blue, thin] (A') -- (B');
\draw[orange, thin] (A') -- (D);
\draw[green, thin] (C) -- (D);
\draw[red, thin] (C) -- (B');

\filldraw[black] (A') circle (12pt) node[anchor=north west]{$D'$};
\filldraw[black] (B') circle (12pt) node[anchor=south west]{$C'$};
\filldraw[black] (C) circle (12pt) node[anchor=south east]{$B'$};
\filldraw[black] (D) circle (12pt) node[anchor=north east]{$A'$};

\draw[purple, name path=line 1, thin] (B') -- (D);
\draw[gray!60, name path=line 2, thin] (A') -- (C);

\end{tikzpicture}
\caption{Dual quadrilaterals. Corresponding sides (with the same color) are parallel and non-corresponding diagonals (also with the same color) are parallel.}
\label{figure:dual-of-2-quads}
\end{figure}

\textbf{(V)} The \emph{mixed area} of two (labeled) plane closed broken lines
$ABCD$ and $A'B'C'D'$ with parallel corresponding sides is
$
\frac{d}{dt}\left.\mathrm{Area}(A(t)B(t)C(t)D(t))\right|_{t=0}
$,
where $A(t):=A+A't$,  $B(t):=B+B't$ etc, and the oriented area is defined by~\eqref{eq-def-curvature}. Note that the mixed area
is invariant under the swap of the two broken lines. 
The \emph{mixed curvature} of two corresponding non-boundary vertices of two v-parallel 
$m\times n$ nets is the mixed area of the top views of the isotropic Gauss images of the vertices. The latter is defined analogously to Definition~\ref{def-curvature}: the \emph{top view of the isotropic Gauss image} of a vertex is the closed broken line $\overline{p_1^*}\,\overline{p_2^*}\,\overline{p_3^*}\,\overline{p_4^*}$, where $p_1,p_2,p_3,p_4$ are the consecutive faces around the vertex. We set by definition that two corresponding boundary vertices always have vanishing mixed curvature.


Here the usage of the notion ``mixed area'' matches \cite{bobenko-2010-curv}. Since the curvature equals a certain area in our setup, we use the abbreviation ``mixed curvature'' for the corresponding mixed area, just like in~\cite[Section~2.4]{isometric-isotropic}.

\textbf{(R)} 
By a (general) \emph{net} we mean any collection of $mn$ points $C_{kl}\in\mathbb{R}^3$ indexed by two integers $0\le k< m$ and $0\le l< n$. If all points $C_{kl}$ coincide, then the net is \emph{trivial}. A nontrivial net
$C_{kl}$ 
is \emph{reciprocal-parallel} to a net $F_{ij}$, where $0\le i\le m$, $0\le j\le n$, 
if 
$C_{i,j-1}C_{ij}\parallel F_{i+1,j}F_{ij}$ for each $0\le i<m$,  $0<j<n$ 
and $C_{i-1,j}C_{ij}\parallel F_{i,j+1}F_{ij}$ for each $0< i<m$, $0\le j< n$. See Figure~\ref{figure:recip-parall}. 

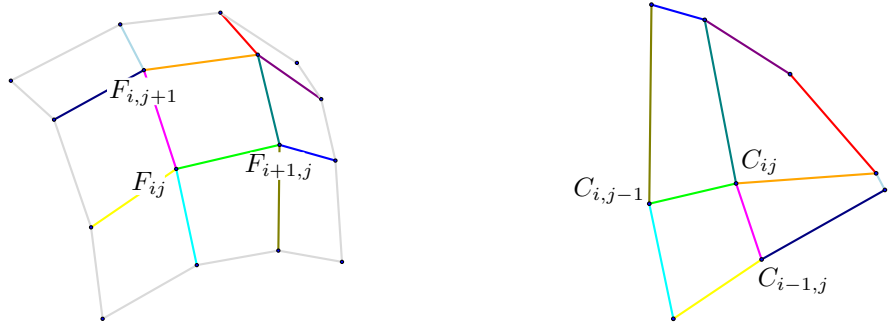
\begin{figure}[htbp]
    \definecolor{color1}{RGB}{255,0,0}    
    \definecolor{color2}{RGB}{0,255,0}    
    \definecolor{color3}{RGB}{0,0,255}    
    \definecolor{color4}{RGB}{255,255,0}  
    \definecolor{color5}{RGB}{255,165,0}  
    \definecolor{color6}{RGB}{128,0,128}  
    \definecolor{color7}{RGB}{0,255,255}  
    \definecolor{color8}{RGB}{255,0,255}  
    \definecolor{color9}{RGB}{173,216,230} 
    \definecolor{color10}{RGB}{128,128,0} 
    \definecolor{color11}{RGB}{0,128,128} 
    \definecolor{color12}{RGB}{0,0,128}   
    
    \centering    
    \begin{tikzpicture}[scale=0.61]
        \coordinate (f03) at (-7.17881, 1.8129);
        \coordinate (f02) at (-6.2473, 0.96606);
        \coordinate (f01) at (-5.44881,-1.37409);
        \coordinate (f00) at (-5.2,-3.36197);
        \coordinate (f13) at (-4.82013, 3.03595);
        \coordinate (f12) at (-4.31006, 2.04267);
        \coordinate (f11) at (-3.61208, -0.10496);
        \coordinate (f10) at (-3.16845,-2.19516);
        \coordinate (f23) at (-2.65907, 3.29098);
        \coordinate (f22) at (-1.85371, 2.37824);
        \coordinate (f21) at (-1.38392, 0.41853);
        \coordinate (f20) at (-1.40955,-1.87961);
        \coordinate (f33) at (-1.00927, 2.2);
        \coordinate (f32) at (-0.4846, 1.4118);
        \coordinate (f31) at (-0.17834, 0.07689);
        \coordinate (f30) at (-0.0372, -2.12487);
        
        \draw[color4, thick] (f11) -- (f01);
        \draw[color2, thick] (f21) -- (f11);
        \draw[color3, thick] (f31) -- (f21);
        \draw[color12, thick] (f12) -- (f02);
        \draw[color5, thick] (f22) -- (f12);
        \draw[color6, thick] (f32) -- (f22);
        \draw[color7, thick] (f11) -- (f10);
        \draw[color8, thick] (f12) -- (f11);
        \draw[color9, thick] (f13) -- (f12);
        \draw[color10, thick] (f21) -- (f20);
        \draw[color11, thick] (f22) -- (f21);
        \draw[color1, thick] (f23) -- (f22);
        
        \draw[gray!30, thick] (f03) -- (f13) -- (f23) -- (f33) -- (f32) -- (f31) -- (f30) -- (f20) -- (f10) -- (f00) -- (f01) -- (f02) -- (f03);
        
        
        \node[draw=white, fill=white, rectangle, inner sep=5pt, yshift=-8pt, xshift=-0.8pt, minimum width=0.9cm, minimum height=0.35cm] at (f12) {};
        \node[draw=white, fill=white, rectangle, inner sep=5pt, yshift=-8pt, xshift=-0.8pt, minimum width=0.7cm, minimum height=0.35cm] at (f21) {};
        \node[draw=white, fill=white, rectangle, inner sep=-1.2pt, yshift=-6pt, xshift=-10pt, minimum width=0cm, minimum height=0cm] at (f11) {$F_{ij}$};
        \filldraw[black] (f12) node[anchor=north]{$F_{i,j+1}$};
        \filldraw[black] (f21) node[anchor=north]{$F_{i+1,j}$};
        
        \foreach \name in {f00,f01,f02,f03,f10,f11,f12,f13,f20,f21,f22,f23,f30,f31,f32,f33} {
            \filldraw[fill=blue] (\name) circle (1pt) node[]{};
        }
    \end{tikzpicture}\qquad\qquad\qquad\qquad
    \begin{tikzpicture}[scale=0.68]

        \coordinate (c02) at (10.29415, -2.36481);
        \coordinate (c01) at (7.9089, -3.71974);
        \coordinate (c00) at (6.201126674523369, -4.877071907148679);
        \coordinate (c12) at (10.12478, -2.04019);
        \coordinate (c11) at (7.41492, -2.23778);
        \coordinate (c10) at (5.73537, -2.63297);
        \coordinate (c22) at (8.45934, -0.10659);
        \coordinate (c21) at (6.80802, 0.95195);
        \coordinate (c20) at (5.77771, 1.24834);
        
        \draw[color4, thick] (c00) -- (c01);
        \draw[color2, thick] (c10) -- (c11);
        \draw[color3, thick] (c20) -- (c21);
        \draw[color12, thick] (c01) -- (c02);
        \draw[color5, thick] (c11) -- (c12);
        \draw[color6, thick] (c21) -- (c22);
        \draw[color7, thick] (c00) -- (c10);
        \draw[color8, thick] (c01) -- (c11);
        \draw[color9, thick] (c02) -- (c12);
        \draw[color10, thick] (c10) -- (c20);
        \draw[color11, thick] (c11) -- (c21);
        \draw[color1, thick] (c12) -- (c22);
        
        
        \filldraw[black] (c11) node[yshift=8pt, xshift=9pt]{$C_{ij}$};
        \filldraw[black] (c01) node[yshift=-8pt, xshift=12pt]{$C_{i-1,j}$};
        \filldraw[black] (c10) node[yshift=5pt, xshift=-15pt]{$C_{i,j-1}$};
        
        \foreach \name in {c00,c01,c02,c10,c11,c12,c20,c21,c22} {
            \filldraw[fill=blue] (\name) circle (1pt) node[]{};
        }
    \end{tikzpicture}
    
    \caption{A reciprocal-parallel net (right) to a $3\times 3$ net (left). The edges of the same color are parallel (except for the light gray ones). See Proposition~\ref{prop-infinitesimal-characterization}. }
    \label{figure:recip-parall}
\end{figure}

\textbf{(E)} An \emph{infinitesimal Euclidean congruence} is a vector field on $\mathbb{R}^3$ of the form $V(\mathbf{x})=\mathbf{a}\times\mathbf{x}+\mathbf{b}$
for some $\mathbf{a},\mathbf{b}\in\mathbb{R}^3$. 
An \emph{infinitesimal Euclidean isometric deformation} of an $m\times n$ net $F_{ij}$
is a net 
$V_{ij}\in\mathbb{R}^3$ 
such that 
for each $0\leq k< m,0\leq l< n$ there is an infinitesimal Euclidean congruence $V$ (possibly, depending on $k$ and $l$) satisfying $V_{ij}=V(F_{ij})$ for each $i\in\{k,k+1\}$ and $j\in\{l,l+1\}$.
An infinitesimal Euclidean isometric deformation is \emph{trivial} if there is an infinitesimal Euclidean congruence $V$ such that $V_{ij}=V(F_{ij})$ for all $0\leq i\leq m,0\leq j\leq n$, otherwise it is \emph{nontrivial}. 
A net is \emph{infinitesimally flexible in Euclidean geometry} if it has a nontrivial infinitesimal Euclidean isometric deformation.

\fixskip
\begin{example} 
    Any dual-convex $m\times n$ net with planar parameter lines 
    is infinitesimally flexible 
    (in both Euclidean and isotropic geometries). 
\end{example}
\fixskip

This follows from Proposition~\ref{prop-infinitesimal-characterization}\textbf{(I,K,E)} and \cite[Theorem~22]{kilian2023smooth}.
See~\cite{PP-2022} for applications of such~nets. 

\subsection{Proof of the classification}
\label{ssec-infinitesimal-proofs}

In this subsection, we prove Proposition~\ref{prop-infinitesimal-characterization}. 
First, we need two lemmas showing that our basic notions are well-defined: 
the metric dual of a dual-convex net has indeed convex faces, and the net used in the vertex condition in Definition~\ref{def-infinitesimal-flexibility} has indeed planar faces.


\fixskip
\begin{lemma} \label{l-dual-convex} Points $A,B,C,D\in I^3$ are consecutive vertices of a convex quadrilateral not lying in an isotropic plane if and only if planes $A^*,B^*,C^*,D^*$ contain consecutive flat angles of an admissible $4$-hedral angle. 
\end{lemma}

\begin{proof} 
Non-collinear points $A,B,C,D$ lie in a non-isotropic plane 
if and only if 
planes $A^*,B^*,C^*,D^*$ have a unique common point $O$. Without loss of generality, assume $O$ is the origin (we can always translate the planes $A^*,B^*,C^*,D^*$ so that they pass through the origin). In this case,  $A,B,C,D$ lie in the $xy$-plane. If $A=(A_1,A_2,0)$ and $B=(B_1,B_2,0)$, then planes $A^*$ and $B^*$ have equations $z=A_1x+A_2y$ and 
$z=B_1x+B_2y$ respectively. Thus the top view of the line $A^*\cap B^*$ is $(A_1-B_1)x+(A_2-B_2)y=0$. 
Denote by $A'=(B_2-A_2,A_1-B_1,A_1B_2-A_2B_1)$ the point on the line $A^*\cap B^*$ with the top view $\overline{A'}=(B_2-A_2,A_1-B_1)$. 
Define 
$B',C',D'$ 
analogously.  See Figure~\ref{figure:lemma-dual-conv}. 

Assume that $ABCD$ is a convex quadrilateral. The vectors $O\overline{A'},\dots, O\overline{D'}$ are obtained by the same clockwise $90^\circ$-rotation from 
$AB,BC,CD,DA$.
Hence $O\overline{A'},\dots, O\overline{D'}$ appear in the listed cyclic order around $O$. Assume without loss of generality that their order is counterclockwise. Since $O\overline{A'}+\dots+O\overline{D'}=0$, it follows that $O$ lies in the interior of the convex hull of $\overline{A'},\overline{B'},\overline{C'},\overline{D'}$. Hence 
the union of angles ${A'}O{B'},\dots,{D'}O{A'}$ is the graph of some function 
$z(x,y)$. Let us show that the function is convex. Let
$A^\varepsilon:=\overline{A'}+\varepsilon\cdot{AB}$, where $\varepsilon>0$, be a small shift of $\overline{A'}$ in the counterclockwise direction about $O$. At the point $A^\varepsilon$, the linear function $z=B_1x+B_2y$ is greater than 
$z=A_1x+A_2y$ because the difference is 
$\varepsilon|{AB}|^2>0$. Since the angles \( D'OA' \) and \( A'OB' \) lie in the planes \( A^* \) and \( B^* \), respectively, their union is convex in the vicinity of \( O\overline{A'} \). Thus, $z(x,y)$ is convex in the vicinity of $O\overline{A'}$. Analogously, 
$z(x,y)$ is convex in the vicinities of $O\overline{B'},\dots, O\overline{D'}$.
Thus the graph of $z(x,y)$ bounds the 
desired admissible $4$-hedral angle.

Conversely, if the planes \(A^*, B^*, C^*, D^*\) contain consecutive flat angles of an admissible 4-hedral angle, then the latter is generated by the vectors $\pm O{A}',\pm O{B}',\pm O{C}',\pm O{D}'$
for some choice of signs. The union of flat angles is the graph of some convex or concave function $z(x,y)$, and the top views of the vectors appear in the listed cyclic order around $O$. Without loss of generality, assume that the order is counterclockwise. We now conclude that all the signs in \(\pm\) must be the same. Indeed,
similar to the above, consider a small counterclockwise shift of $\pm\overline{A'}$, say, 
$A^\varepsilon:=\pm\overline{A'}\pm\varepsilon\cdot{AB}$ for some $\varepsilon>0$.
At the point \(A^\varepsilon\), the linear function 
$z=B_1x+B_2y$ is greater (respectively, less) than 
$z=A_1x+A_2y$ 
if \(z(x,y)\) is convex 
(respectively, concave). Hence the signs in 
$\pm O{A}',\dots,\pm O{D}'$
must be the same.
Since the vectors $AB,BC,CD,DA$ are obtained by the same $90^\circ$
-rotation from $O\overline{A'},\dots, O\overline{D'}$, the points \(A, B, C, D\) are consecutive vertices of a convex quadrilateral.
\end{proof}

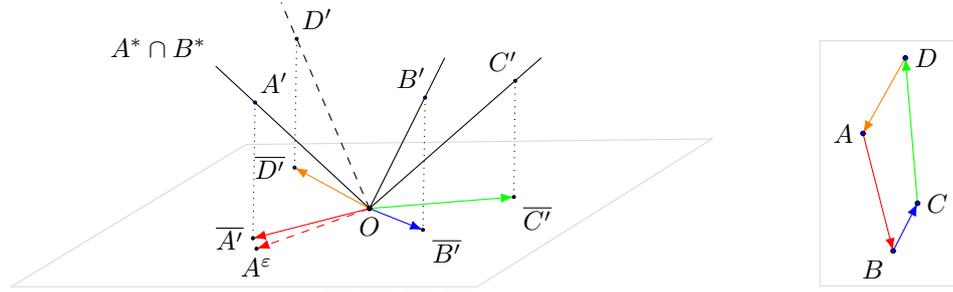
\begin{figure}[htbp]
  \centering    
\begin{tikzpicture}[scale=0.56]

\coordinate (p1) at (1.58474, -36.8);
\coordinate (p2) at (7.10107, -33.43342);
\coordinate (p3) at (18.0531, -33.30089);
\coordinate (p4) at (12.53677, -36.8);
\coordinate (O) at (10.00261, -34.94882);
\coordinate (AB) at (6.38876, -31.57807);
\coordinate (BC) at (7.92935, -30.07061);
\coordinate (CD) at (14.04204, -31.37929);
\coordinate (DA) at (11.77256, -31.37929);
\coordinate (A') at (7.31643, -32.43948);
\coordinate (B') at (8.29379, -30.93202);
\coordinate (C') at (13.42912, -31.92595);
\coordinate (D') at (11.30873, -32.32352);
\coordinate (A0) at (7.26673, -35.6532);
\coordinate (A1) at (7.35, -35.9);
\coordinate (B0) at (8.2441, -33.98008);
\coordinate (C0) at (13.39599, -34.67583);
\coordinate (D0) at (11.25903, -35.45442);

\draw[gray!30, thin] (p1) -- (p2) -- (p3) -- (p4) -- (p1);

\draw[black, thin] (AB) -- (O);
\draw[black, thin, dashed] (O) -- (BC);
\draw[black, thin] (CD) -- (O) -- (DA);
\draw[black, thin, dotted] (A') -- (A0);
\draw[black, thin, dotted] (B') -- (B0);
\draw[black, thin, dotted] (C') -- (C0);
\draw[black, thin, dotted] (D') -- (D0);

\draw[red, thin, -Latex] (O) -- (A0);
\draw[red, dashed, -Latex] (O) -- (A1);
\draw[orange, thin, -Latex] (O) -- (B0);
\draw[green, thin, -Latex] (O) -- (C0);
\draw[blue, thin, -Latex] (O) -- (D0);

\filldraw[black] (7.21643, -32.43948) node[anchor=south west]{$A'$};
\filldraw[black] (A1) circle (1pt) node[anchor=north]{$A^\varepsilon$};
\filldraw[black] (8.19379, -30.93202) node[anchor=south west]{$D'$};
\filldraw[black] (13.7, -31.92595) node[anchor=south east]{$C'$};
\filldraw[black] (11.55, -32.32352) node[anchor=south east]{$B'$};

\filldraw[black] (A0) node[anchor=east]{$\overline{A'}$};
\filldraw[black] (B0) node[anchor=east]{$\overline{D'}$};
\filldraw[black] (C0) node[anchor=north west]{$\overline{C'}$};
\filldraw[black] (D0) node[anchor=north west]{$\overline{B'}$};

\filldraw[black] (O) circle (1.5pt) node[anchor=north]{$O$};
\filldraw[black] (AB) node[anchor=south east]{$A^*\cap B^*$};

\foreach \name in {A',B',C',D',A0,B0,C0,D0} {
    \filldraw[fill=blue] (\name) circle (1pt) node[]{};
}
\end{tikzpicture}\qquad\qquad
\begin{tikzpicture}[scale=0.57]

\coordinate (p1) at (15.5, -39);
\coordinate (p2) at (15.5, -33.3);
\coordinate (p3) at (18.7, -33.3);
\coordinate (p4) at (18.7, -39);
\coordinate (A) at (16.49375, -35.45442);
\coordinate (B) at (17.1895, -38.18773);
\coordinate (C) at (17.75273, -37.07784);
\coordinate (D) at (17.47111, -33.69847);

\draw[gray!30, thin] (p1) -- (p2) -- (p3) -- (p4) -- (p1);

\draw[red, thin, -Latex] (A) -- (B);
\draw[blue, thin, -Latex] (B) -- (C);
\draw[green, thin, -Latex] (C) -- (D);
\draw[orange, thin, -Latex] (D) -- (A);

\filldraw[black] (A) node[anchor=east]{$A$};
\filldraw[black] (B) node[anchor=north east]{$B$};
\filldraw[black] (C) node[anchor=west]{$C$};
\filldraw[black] (D) node[anchor=west]{$D$};

\foreach \name in {A,B,C,D} {
    \filldraw[fill=blue] (\name) circle (1.5pt) node[]{};
}
\end{tikzpicture}

\caption{Left: The points \(\overline{A'}, \overline{B'}, \overline{C'}, \overline{D'}\) are the top views of \(A', B', C', D'\), respectively. The vectors \(O\overline{A'}, O\overline{B'}, O\overline{C'}, O\overline{D'}\) are obtained from \(AB, BC, CD, DA\) (right) by the 
clockwise rotation through \(90^\circ\) about the origin $O$. 
The point $A^\varepsilon$ is a small counterclockwise shift of $\overline{A'}$.
Right: The 
side vectors of the quadrilateral $ABCD$. See the proof of Lemma~\ref{l-dual-convex}.}
\label{figure:lemma-dual-conv}
\end{figure}


\begin{lemma} \label{l-affine} 
If $F_{ij}$ is a dual-convex $m\times n$ net and $V_{ij}$ is a collection of vectors satisfying the face condition from Definition~\ref{def-infinitesimal-flexibility}, then $F_{ij}+tV_{ij}$ is also a dual-convex $m\times n$ net for all $t$ sufficiently
small in absolute value. 
\end{lemma}

\begin{proof} Consider a face $p_{kl}$ of the net $F_{ij}$. Take a vertex $F_{ij}$ of this face. By the face condition in Definition~\ref{def-infinitesimal-flexibility}, we get  $F_{ij}+tV_{ij}=F_{ij}+tV(F_{ij})$ for some infinitesimal isotropic congruence $V$. Then 
the affine transformation $\mathbf{x}\mapsto \mathbf{x}+tV(\mathbf{x})$
takes $p_{kl}$ to the corresponding face of the net $F_{ij}+tV_{ij}$. Thus the latter face is a convex quadrilateral. 
By the continuity, the net $F_{ij}+tV_{ij}$ remains dual-convex for sufficiently small $|t|$.
\end{proof}
   
Now we proceed to the proof of Proposition~\ref{prop-infinitesimal-characterization} and first address the equivalence \textbf{(I)}$\Leftrightarrow$\textbf{(V)}. The following lemma allows us to restrict ourselves to isotropic vectors $V_{ij}$. 

\fixskip
\begin{lemma} \label{l-isotropic-infinitesimal} An infinitesimally flexible dual-convex $m\times n$ net in $I^3$ 
has a nontrivial infinitesimal isotropic isometric deformation consisting of isotropic vectors. 
\end{lemma}

\begin{proof} Let $V_{ij}$, where $0\leq i\leq m,0\leq j\leq n$, be a nontrivial infinitesimal isotropic isometric deformation of the $m\times n$ 
net.
Then the top view $\overline{V_{ij}}$ 
is an infinitesimal Euclidean isometric deformation of
the top view of the net because the top view of each infinitesimal isotropic congruence $V$ (in face condition of Definition~\ref{def-infinitesimal-flexibility}) is an infinitesimal Euclidean congruence. 

Each $m\times n$ net in the plane is infinitesimally rigid, i.e., all planar infinitesimal Euclidean isometries are trivial. (Indeed, let \(V\colon\mathbf{x} \mapsto a_{kl} \times \mathbf{x} + b_{kl}\), where vectors \(a_{kl} \parallel Oz\) and \(b_{kl} \parallel Oxy\), be a planar infinitesimal Euclidean congruence for a face \(p_{kl}\) such that $\overline{V_{ij}}=V(F_{ij})$ for all vertices $F_{ij}$ of \(p_{kl}\). Comparing the images of the common points of neighboring faces under the maps $V$, we obtain that \( a_{kl} \) forms a reciprocal-parallel net of \( F_{ij} \). This implies that all \(a_{kl}\) are equal and consequently all \(b_{kl}\) are.) 
Then $\overline{V_{ij}}$ is a trivial infinitesimal Euclidean isometric deformation. Then $\overline{V_{ij}}$ can be viewed as a trivial infinitesimal isotropic isometric deformation. Then $V_{ij}-\overline{V_{ij}}$ is the desired nontrivial infinitesimal isotropic isometric deformation.  
\end{proof}

If a nontrivial infinitesimal isotropic isometric deformation $V_{ij}$ of {a dual-convex} 
$m\times n$ net $F_{ij}$ consists of isotropic vectors, 
then the net $\overline{F_{ij}}+V_{ij}$ is called a \emph{velocity diagram} of the net $F_{ij}$. (This definition differs from the one in \cite{isometric-isotropic} by a rotation through $\pi/2$ about the $z$-axis.) By Lemma~\ref{l-isotropic-infinitesimal}, the net $F_{ij}$ has a velocity diagram if and only if $F_{ij}$ is infinitesimally flexible in $I^3$.

\fixskip

\begin{lemma} \label{l-flexible-mixed} A net is a velocity diagram of 
a dual-convex $m\times n$ net if and only if it is a non-planar v-parallel $m\times n$ net with vanishing mixed curvatures at the corresponding vertices.
\end{lemma}

\begin{proof} 
Consider a dual-convex $m\times n$ net $F_{ij}$ with faces $p_{kl}$. Denote by $(F_{ij}^1,F_{ij}^2,F_{ij}^3)$ and $(p_{kl}^1,p_{kl}^2,p_{kl}^3)$ the coordinates of vertices $F_{ij}$ and $p_{kl}^*$ of the net and its metric dual. 

Assume that the net $F_{ij}$ has a velocity diagram. 
Let $V_{ij}$ be its nontrivial infinitesimal isotropic isometric deformation consisting of isotropic vectors. Let us show that the velocity diagram $\overline{F_{ij}}+V_{ij}$ is a non-planar v-parallel $m\times n$ net with vanishing mixed curvatures at the corresponding vertices.
Clearly, the velocity diagram is v-parallel. 

Let us show that the velocity diagram has planar faces. Consider a face $p_{kl}$ of the net $F_{ij}$. 
By the face condition in Definition~\ref{def-infinitesimal-flexibility}, there is an infinitesimal isotropic congruence $V$ such that $V_{ij}=V(F_{ij})$ for each $i\in\{k,k+1\}$ and $j\in\{l,l+1\}$. Since $V_{ij}\parallel Oz\nparallel p_{kl}$, by~\eqref{eq-isotropic-congruence} it follows that $V(x,y,z)=(0,0,c_{kl}^1 x+c_{kl}^2 y-c_{kl}^3)$ for some $c_{kl}^1,c_{kl}^2,c_{kl}^3\in\mathbb{R}$. This implies that the four points $\overline{F_{ij}}+V_{ij}$ lie in the plane $z=c_{kl}^1 x+c_{kl}^2 y-c_{kl}^3$ and hence the velocity diagram $\overline{F_{ij}}+V_{ij}$ has planar faces. Thus it is an $m\times n$ net, and the point $c^*_{kl}:=(c_{kl}^1,c_{kl}^2,c_{kl}^3)$ is metric dual to the plane of its face. The $m\times n$ net is non-planar, otherwise all $c^*_{kl}$ are equal and $V_{ij}$ is trivial.

Now we show that the mixed curvatures vanish at the corresponding vertices. From Lemma~\ref{l-affine} we know that $F_{ij}+tV_{ij}$ is a dual-convex $m\times n$ net for all sufficiently small $t$. Let us find its metric dual. The four points $F_{ij}+tV_{ij}$, where $i\in\{k,k+1\}$ and $j\in\{l,l+1\}$, lie in one plane $z=(p_{kl}^1+t c_{kl}^1)x+(p_{kl}^2+tc_{kl}^2) y-(p_{kl}^3+tc_{kl}^3)$. Hence the net $p^*_{kl}(t):=p^*_{kl}+tc^*_{kl}$ is metric dual to the net $F_{ij}+tV_{ij}$. Since the nets $F_{ij}+tV_{ij}$ are all $v$-parallel, it follows that the nets $p^*_{kl}(t)$ are parallel 
(i.e., Combescure transformations). In particular, the nets $c^*_{kl}$ and $p^*_{kl}$ are parallel. By the vertex condition of Definition~\ref{def-infinitesimal-flexibility},
$$\frac{d}{dt}\left.\mathrm{Area}\left(\overline{p^*_{ij}(t)}\,\overline{p^*_{i-1,j}(t)}\,\overline{p^*_{i-1,j-1}(t)}\,\overline{p^*_{i,j-1}(t)}\right)\right|_{t=0}=\frac{d}{dt}\left.\Omega(F_{ij}+tV_{ij})\right|_{t=0}=0.$$ Thus 
the mixed areas of the corresponding faces of the nets $p^*_{kl}$ and $c^*_{kl}$ vanish. Hence the mixed curvature of $F_{ij}$ and $\overline{F_{ij}}+V_{ij}$ vanishes.

Conversely, let $F'_{ij}$ be a non-planar $m\times n$ net v-parallel to 
a dual-convex $m\times n$ net $F_{ij}$ with vanishing mixed curvatures at the corresponding vertices. Then the metric dual net $c^*_{kl}=(c_{kl}^1,c_{kl}^2,c_{kl}^3)$ of $F'_{ij}$ is a nontrivial infinitesimal deformation of the net $p^*_{kl}$ (see definition \textbf{(D)} after  Proposition~\ref{prop-infinitesimal-characterization}). Set $V_{ij}:=F_{ij}'-\overline{F_{ij}}=(0,0,c_{kl}^1F_{ij}^1+c_{kl}^2F_{ij}^2-c_{kl}^3)$ for each $0\le i\le m$ and $0\le j\le n$, where $c_{kl}$ is any face containing the vertex $F'_{ij}$. Here the latter equality holds because 
the plane of the face $c_{kl}$ 
has the equation $z=c_{kl}^1x+c_{kl}^2y-c_{kl}^3$ so that $F_{ij}'=(F_{ij}^1,F_{ij}^2,c_{kl}^1F_{ij}^1+c_{kl}^2F_{ij}^2-c_{kl}^3)$ and 
$F_{ij}'-\overline{F_{ij}}=(0,0,c_{kl}^1F_{ij}^1+c_{kl}^2F_{ij}^2-c_{kl}^3)$. 
The vectors $V_{ij}$ satisfy the face condition because the infinitesimal isotropic congruence $V(x,y,z)=(0,0,c_{kl}^1 x+c_{kl}^2 y-c_{kl}^3)$ satisfies $V_{ij}=V(F_{ij})$ for each $i\in\{k,k+1\}$ and $j\in\{l,l+1\}$. The vectors $V_{ij}$ satisfy the vertex condition because the net $p_{kl}^*+tc_{kl}^*$ is metric dual to $F_{ij}+tV_{ij}$ and the vectors $c_{kl}^*$ form an infinitesimal deformation of $p_{kl}^*$. The vectors $V_{ij}$ form a nontrivial infinitesimal isotropic isometric deformation because not all points $c^*_{kl}:=(c_{kl}^1,c_{kl}^2,c_{kl}^3)$ coincide.
Thus $F_{ij}'$ is a velocity diagram of $F_{ij}$.
\end{proof}

To prove the equivalence \textbf{(V)}$\Leftrightarrow$\textbf{(D)}, we just apply the metric duality. We observe that the metric duality takes v-parallel nets to parallel nets. More precisely, a net $C_{kl}$ is parallel to the metric dual of a dual-convex $m\times n$ net $F_{ij}$ if and only if $C_{kl}$ is the metric dual of some $m\times n$ net that is v-parallel to $F_{ij}$. The following lemma is directly obtained from the definitions.

\fixskip
\begin{lemma} \label{l-mixed-deformable} 
A non-planar $m\times n$ net, v-parallel to a dual-convex $m\times n$ net, has vanishing mixed curvatures at the corresponding vertices if and only if 
the metric dual of the former is a nontrivial infinitesimal deformation
of the metric dual of the latter.
\end{lemma}



\fixskip
To prove the equivalence \textbf{(D)}$\Leftrightarrow$\textbf{(K)}, let us recall the characterization of infinitesimally deformable $m\times n$ nets from~\cite[Proposition~37]{isometric-isotropic}.

\fixskip
\begin{lemma} \label{l-infinitesimal-deformable}
An $(m-1)\times (n-1)$ net is infinitesimally deformable if and only if it is a K\oe nings net. Its nontrivial infinitesimal deformations 
are exactly Christoffel dual $(m-1)\times (n-1)$~nets.
\end{lemma}

\begin{proof}
    Let vectors $C_{ij}$, 
    which do not all coincide, form an infinitesimal deformation of an $(m-1)\times (n-1)$ net $P_{ij}$. 
    Set $P_{ij}(t):=P_{ij}+tC_{ij}$. The quadrilaterals $P_{ij}(t)P_{i+1,j}(t)P_{i+1,j+1}(t)P_{i,j+1}(t)$ have parallel sides and the derivative of the area at $t=0$ vanishes. By a known computation \cite[Theorem~13 and Eq.~(3)]{bobenko-2010-curv}, this is equivalent to quadrilaterals
    $P_{ij}P_{i+1,j}P_{i+1,j+1}P_{i,j+1}$ and $C_{ij}C_{i+1,j}C_{i+1,j+1}C_{i,j+1}$ being dual. 
    The latter quadrilateral is convex because the former is (and cannot degenerate into a point, unless the whole net $C_{ij}$ does). Thus $P_{ij}$ and $C_{ij}$ are Christoffel dual $(m-1)\times (n-1)$ nets unless all vectors $C_{ij}$ are equal.
\end{proof}

In particular, any infinitesimal deformation with not all vectors equal has convex faces, 
thus by Lemmas~\ref{l-affine}, \ref{l-flexible-mixed}, and \ref{l-mixed-deformable}, any velocity diagram is 
a dual-convex $m\times n$ net.



To prove the equivalence \textbf{(K)}$\Leftrightarrow$\textbf{(R)}, we need the following property of planar Christoffel dual nets, with not necessarily convex faces. For such nets, the Christoffel duality is defined analogously to the above. Clearly, if one of two nontrivial Christoffel dual nets has convex faces, then the second one does.

\fixskip
\begin{lemma} \label{l-height-function}
    Two parallel 
    nets ${p}_{kl}$ and ${c}_{kl}$ in the plane are Christoffel dual if and only if there is an indexed collection of real numbers $h_{kl}$ such that for any adjacent vertices ${c}_{kl}$ and ${c}_{k'l'}$ we have 
    \begin{equation}\label{eq-height-function}
       h_{kl}-h_{k'l'}=(c_{kl}^1-c_{k'l'}^1)p_{kl}^2-(c_{kl}^2-c_{k'l'}^2)p_{kl}^1.
    \end{equation}
\end{lemma}

\fixskip
Notice that the collection $h_{kl}$ is determined by
~\eqref{eq-height-function} uniquely up to a 
constant.

\begin{proof}
    Consider a face of the net ${c}_{kl}$. 
    Denote by ${c}_1,{c}_2,{c}_3,{c}_4$    
    its four consecutive vertices. Set ${c}_{5}:={c}_{1}$ and define ${p}_1,p_{2}\dots,p_5=p_1$ analogously.
    Since the nets ${c}_{kl}$ and ${p}_{kl}$ are parallel, the right side of~\eqref{eq-height-function} 
    can be written in two ways:
    $$
    \det( c_i- c_{i+1}, p_{i})
    = \det( c_i- c_{i+1}, p_{i+1})
    $$
    for suitable $i$. The existence of the desired collection $h_{kl}$ is equivalent to the vanishing~of
    \begin{multline*}
        \sum_{i=1}^{4}\det( c_{i}- c_{i+1}, p_{i})
        =
        \det( c_{1}- c_{2}, p_{2})+ \det( c_{2}- c_{3}, p_{2})
        +\det( c_{3}- c_{4}, p_{4})+\\+ \det( c_{4}- c_{1}, p_{4})
        =\det( c_{1}- c_{3}, p_{2}-p_{4}).
    \end{multline*}
    The vanishing of the latter is equivalent to the Christoffel duality.
\end{proof}

Now we construct the reciprocal-parallel net explicitly.

\fixskip
\begin{lemma} \label{l-deformable-reciprocal}
Let $F_{ij}$ be a dual-convex $m\times n$ net and $p^*_{kl}$ be its metric dual. 
A net $C_{kl}$ 
is reciprocal-parallel to the $m\times n$ net~$F_{ij}$
if and only if $C_{kl}=(-c_{kl}^2,c_{kl}^1,h_{kl})$, where $c^*_{kl}=(c_{kl}^1,c_{kl}^2,c_{kl}^3)$ is a (nontrivial)  Christoffel dual to 
$p^*_{kl}$, and $h_{kl}$ is given by~\eqref{eq-height-function}.
\end{lemma}

\begin{proof}
Denote $F_{ij}=:(F_{ij}^1,F_{ij}^2,F_{ij}^3)$, $p^*_{kl}=:(p_{kl}^1,p_{kl}^2,p_{kl}^3)$, and $C_{kl}=:(C_{kl}^1,C_{kl}^2,C_{kl}^3)$.  
Since $F_{ij}$ and $F_{i+1,j}$ are common vertices of the faces $p_{ij}$ and $p_{i,j-1}$, it follows that 
$$\Scale[0.89]{p_{i,j-1}^1(F_{i+1,j}^1-F_{ij}^1)+p_{i,j-1}^2(F_{i+1,j}^2-F_{ij}^2)=F_{i+1,j}^3-F_{ij}^3=
p_{ij}^1(F_{i+1,j}^1-F_{ij}^1)+p_{ij}^2(F_{i+1,j}^2-F_{ij}^2).}$$ 
The latter implies $\overline{F_{i+1,j}}-\overline{F_{ij}}\perp \overline{p^*_{ij}}-\overline{p^*_{i,j-1}}$.
Thus $C_{i,j-1}-C_{ij}\parallel F_{i+1,j}-F_{ij}$
if and only if $\overline{C_{i,j-1}}-\overline{C_{ij}}\perp \overline{p^*_{ij}}-\overline{p^*_{i,j-1}}$ and 
$
C_{ij}^3-C_{i,j-1}^3=
p_{ij}^1(C_{ij}^1-C_{i,j-1}^1)+p_{ij}^2(C_{ij}^2-C_{i,j-1}^2)$.
The former condition (and the one with the indices swapped) means that the planar net $\overline{c^*_{ij}}:=(C_{ij}^2,-C_{ij}^1)$ is parallel to the net $\overline{p^*_{ij}}$, and the latter condition means that 
$h_{ij}:=C_{ij}^3$ satisfies equation~\eqref{eq-height-function}. By Lemma~\ref{l-height-function}, the two conditions together are equivalent to $\overline{c^*_{ij}}$ and $\overline{p^*_{ij}}$ being Christoffel dual. In the latter case, there is a unique (up to translation) net $c^*_{ij}$ with the top view $\overline{c^*_{ij}}$ such that $c^*_{ij}$ and $p^*_{ij}$ are Christoffel dual (parallelism of faces implies their duality in this case). So, the net $C_{kl}$ is reciprocal-parallel to the $m\times n$ net~$F_{ij}$ if and only if it has the form $C_{kl}=(-c_{kl}^2,c_{kl}^1,h_{kl})$.
\end{proof}


\begin{corollary} \label{p-reciprocal-parallel-top-view}
(Cf.~\cite[Lemma~62]{isometric-isotropic})
   For a 
   dual-convex $m\times n$ net, 
   the top views of the reciprocal-parallel net (if it exists) and the metric dual net are related by 
   a rotation through $\pi/2$ 
   followed by the Christoffel duality.
\end{corollary}


\fixskip
Finally, we summarize the proof.

\begin{proof}[Proof of Proposition~\ref{prop-infinitesimal-characterization}] 
The equivalences \textbf{(I)}$\Leftrightarrow$\textbf{(V)}$\Leftrightarrow$\textbf{(D)}$\Leftrightarrow $\textbf{(K)}$\Leftrightarrow $\textbf{(R)}$\Leftrightarrow $\textbf{(E)} were proved in Lemmas~\ref{l-flexible-mixed}, \ref{l-mixed-deformable}, \ref{l-infinitesimal-deformable}, \ref{l-deformable-reciprocal}, and \cite[pp. 172-174]{sauer:1970} respectively. 
\end{proof}


\section{Finite flexibility} 
\label{sec-finite}

Now we turn to finite flexibility. Again, we impose an additional \emph{vertex condition}, which is automatic in Euclidean geometry but becomes crucial in isotropic geometry.

\subsection{Definition} 

\begin{definition} \label{def-flexibility}
An \emph{isotropic isometric deformation} of a dual-convex $m\times n$ net $F_{ij}$
is a continuous family of $m\times n$ nets $F_{ij}(t)$, where 
$t\in [0;1]$
and $F_{ij}(0)=F_{ij}$, satisfying 
\fixskip
\begin{description}
    \item[\textbf{face condition:}] corresponding faces of all $m\times n$ nets $F_{ij}(t)$ are isotropically congruent; and
    \item[\textbf{vertex condition:}] corresponding non-boundary vertices of all 
    nets $F_{ij}(t)$ have the same curvatures. 
\end{description}
\fixskip
An isotropic isometric deformation is \emph{trivial} if for each 
$t\in [0;1]$
there is an isotropic congruence $C_t$ such that $F_{ij}(t)=C_t(F_{ij})$ for all $0\leq i\leq m,0\leq j\leq n$. 
A dual-convex $m\times n$ net is \emph{flexible} in $I^3$ if it has a nontrivial isotropic isometric deformation.
\end{definition}
\fixskip

We may assume without loss of generality that the isotropic congruence in the face condition preserves the indices of the vertices. 
Also, note that two 
quadrilaterals with the same top view (but not in one isotropic plane) are always isotropically congruent.

\fixskip
\begin{example} \label{ex-2x2} (See Figure~\ref{fig:2x2-isot-def}.) The dual-convex $2\times2$ net with the vertices $${F_{ij}}=(i,j,i\!\!\!\! \mod 2+j\!\!\!\!\mod 2), \qquad
0\leq i,j\leq2,$$ is flexible in $I^3$.
A nontrivial isotropic isometric deformation is given by 
\begin{equation*}
    F_{ij}(t)=
    \left(i,j,(i\!\!\!\!\mod2)\dfrac{1}{{t+1}}+(j\!\!\!\!\mod2)(t+1)\right),        \qquad 0\leq i,j\leq 2, \quad t\in[0,1].
\end{equation*}
\end{example}

\begin{figure}[h]
    \centering
    \includegraphics[scale=0.14]{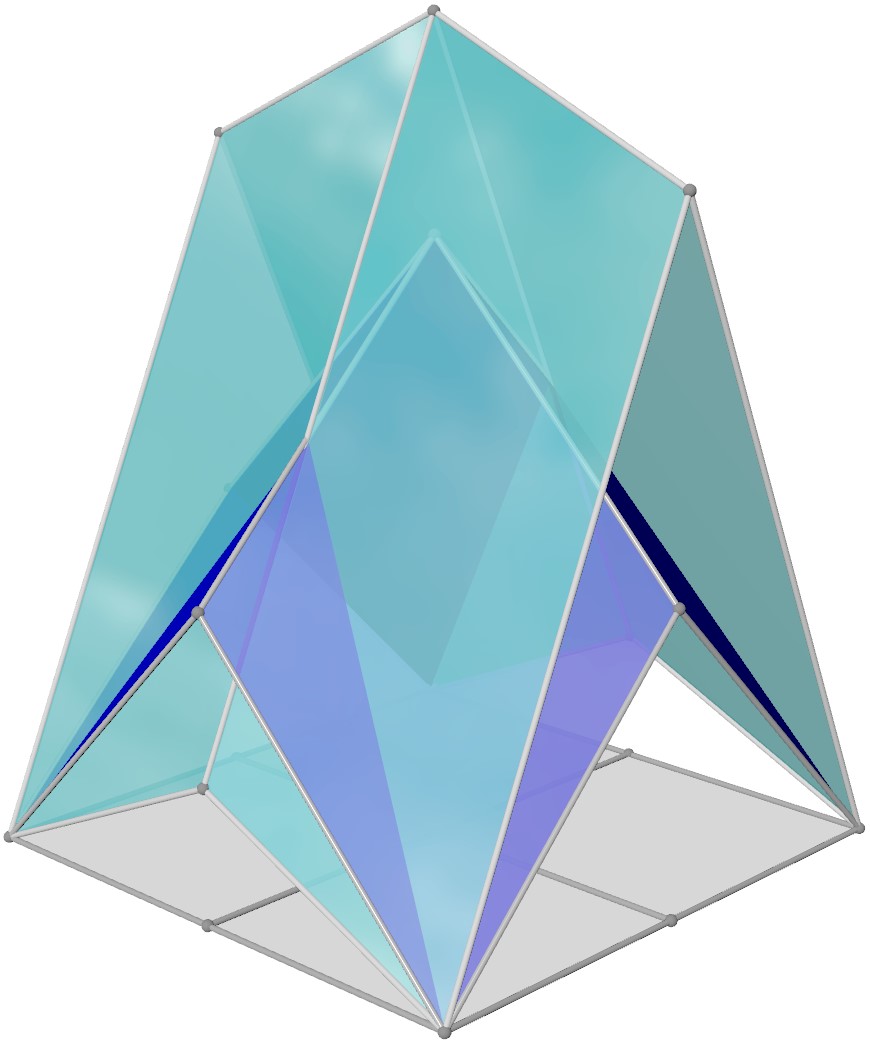}
    \caption{A dual-convex $2\times2$ net that is flexible in $I^3$ (blue). One position of its isotropic isometric deformation (cyan) and their common top view (gray).
    See 
    Example~\ref{ex-2x2}.}
    \label{fig:2x2-isot-def}
\end{figure}

\subsection{Statement of the classification}
\label{ssec-finite-statements}

To characterize all flexible $m\times n$ nets, we need the following notions. 

By an $a\times b$ \emph{sub-net} of a given $m\times n$ net with the points $F_{ij}$, where $0\leq i\leq m$ and $0\leq j\leq n$, we mean an $a\times b$ net with the points $F_{ij}$, where $q\leq i\leq q+a$ and $r\leq j\leq r+b$, for some integers $0\leq q\leq m-a$ and $0\leq r\leq n-b$. (The collection of $(a+1)(b+1)$ points indexed in this way is still viewed as an $a\times b$ net.)

Let $F_{ij}$ be a non-boundary vertex of a dual-convex $m\times n$ net and 
let $e$ be an edge emanating from the vertex. See Figure~\ref{fig:opposite-ratio}. 
Enumerate the consecutive faces $p_1,p_2,p_3,p_4$ around the vertex so that $p_1\cap p_4=e$. 
Let $p$ be the plane spanned by the lines $p_1\cap p_3$ and $p_2\cap p_4$.
The ratio of isotropic angles (see Section~\ref{ssec:isotropic} for the definition)
$$\frac{\angle(p_1,p)}{\angle(p_3,p)}\cdot\frac{\angle(p_4,p)}{\angle(p_2,p)}$$
is called the \emph{opposite ratio of the vertex $F_{ij}$ with respect to the edge $e$}. Clearly, it depends only on the $4$-hedral angle spanned by the faces. 

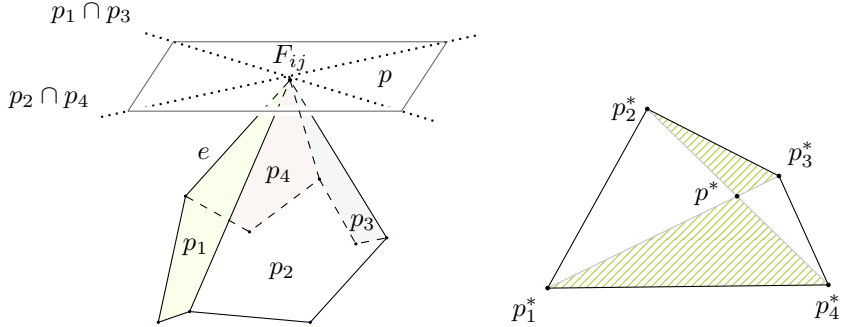
\begin{figure}[h]
    \centering
    \begin{tikzpicture}[scale=0.40]
        \coordinate (A) at (-3.66577, 2.03);
        \coordinate (B) at (-7.10198, -1.81635);
        \coordinate (C) at (-6.96776, -5.64182);
        \coordinate (D) at (-0.4846, -3.19889);
        \coordinate (E) at (-2.7, -1.2526);
        \coordinate (K) at (-5, -3);
        \coordinate (L) at (-8, -6);
        \coordinate (M) at (-3, -6);
        \coordinate (N) at (-1.5, -3.4);
        \coordinate (P) at (-9, 1);
        \coordinate (Q) at (-7.5, 3.3);
        \coordinate (R) at (1.5, 3.3);
        \coordinate (S) at (0, 1);
    
        \def\t{1.14} 
        \def\s{1.09} 
    
        \coordinate (P') at ($(P)!\s!(R)$);
        \coordinate (R') at ($(R)!\s!(P)$);
        \coordinate (S') at ($(S)!\t!(Q)$);
        \coordinate (Q') at ($(Q)!\t!(S)$);
    
        \filldraw[black] (R') node[anchor=south east]{$p_2\cap p_4$};
        \filldraw[black] (S') node[anchor=south east]{$p_1\cap p_3$};
    
        \draw[black, dotted, thick] (P') -- (R');
        \draw[black, dotted, thick] (Q') -- (S');
    
        \fill[britishracinggreen!4] (A) -- (E) -- (N) -- (D) -- cycle;
        \fill[antiquebrass!8] (A) -- (B) -- (K) -- (E) -- cycle;
        \fill[lime!8] (A) -- (B) -- (L) -- (C) -- cycle;
        
        \path[name path=P--S] (P) -- (S);
        \path[name path=A--B] (A) -- (B);
        \path[name path=A--C] (A) -- (C);
        \path[name path=A--D] (A) -- (D);
        
        \path[name intersections={of=P--S and A--B, by={PS-AB}}];
        \path[name intersections={of=P--S and A--C, by={PS-AC}}];
        \path[name intersections={of=P--S and A--D, by={PS-AD}}];
    
        \draw[black, thin] (PS-AB) -- (B) -- (L) -- (C) -- (M) -- (D) -- (PS-AD);
        \draw[black, thin, dashed] (PS-AB) -- (A) -- (PS-AD);
        \draw[black, thin, dashed] (PS-AC) -- (A);
        \draw[black, thin] (PS-AC) -- (C);
        \draw[white, line width=3.5pt] (P) -- (S);
        \draw[black!60, thin] (P) -- (Q) -- (R) -- (S) -- (P);
    
        \draw[black, dashed] (A) -- (E) -- (K) -- (B);
        \draw[black, dashed] (E) -- (N) -- (D);
        
        \filldraw[black] (A) circle (1.5pt) node[anchor=south]{$F_{ij}$};
        \draw[black] (-3.3, -1.7) node[anchor=south east]{$p_4$};
        \draw[black] (-7.5, -3.5) node[anchor=west]{$p_1$};
        \draw[black] (-3.2, -4.3) node[anchor=east]{$p_2$};
        \draw[black] (-2.0, -2.7) node[anchor=west]{$p_3$};
        \draw[black] (0, 1.5) node[anchor=south east]{$p$};
        \draw[black] (-6.5,-0.9) node[anchor=south]{$e$};

        \foreach \name in {B,C,D,E,K,L,M,N} {
            \filldraw[fill=blue] (\name) circle (1pt) node[]{};
        }
    
    \end{tikzpicture}\quad
    \begin{tikzpicture}[scale=1.34]
        \coordinate (A) at (-3.82, -0.74);
        \coordinate (D) at (-2.8416, 1.03277);
        \coordinate (S) at (-2.29447, 2.02757);
        \coordinate (B) at (-1.06341, -0.70812);
        \coordinate (R) at (0.50339, -0.68325);
        
        \path[name path=line 1] (S) -- (B);
        \path[name path=line 2] (D) -- (R);
        \fill[name intersections={of=line 1 and line 2, by=C}];
        
        \path[name path=line 3] (D) -- (B);
        \path[name path=line 4] (A) -- (C);
        \fill[name intersections={of=line 3 and line 4, by=Q}];
        
        \fill[pattern=north east lines, pattern color=applegreen!70] (D) -- (Q) -- (C) -- cycle;
        \fill[pattern=north east lines, pattern color=applegreen!70] (A) -- (Q) -- (B) -- cycle;
        
        \draw[black] (D) -- (A) -- (B) -- (C) -- (D);
        \draw[gray!50, very thin] (D) -- (B);
        \draw[gray!50, very thin] (A) -- (C);
        
        \filldraw[black] (A) circle (0.5pt) node[anchor=north east]{$p_1^*$};
        \filldraw[black] (B) circle (0.5pt) node[anchor=north]{$p_4^*$};
        \filldraw[black] (C) circle (0.5pt) node[anchor=south west]{$p_3^*$};
        \filldraw[black] (D) circle (0.5pt) node[anchor=east]{$p_2^*$};
        \filldraw[black] (Q) circle (0.5pt); \filldraw[black] (-2.03,0.2) circle (0pt) node[anchor=east]{$p^*$};
        
    \end{tikzpicture}
    \caption{Left: The opposite ratio of a vertex $F_{ij}$ with respect to an edge $e$ is defined in terms of the isotropic angles between the plane $p$ (the span of $p_1\cap p_3$ and $p_2\cap p_4$) and the planes $p_1,p_2,p_3,p_4$ of the faces. Right: The opposite ratio of the quadrilateral $p_1^*p_2^*p_3^*p_4^*$ with respect to the side $p_1^*p_4^*$ is the ratio of the areas of the colored triangles $p^*p_1^*p_4^*$ and $p^*p_2^*p_3^*$. See Section~\ref{ssec-finite-statements}.}
    \label{fig:opposite-ratio}
\end{figure}

The following theorem characterizes 
all flexible $m\times n$ nets in $I^3$. See Figure~\ref{fig:flexible-class-i-def}.

\fixskip
\begin{theorem}  \label{th-flexible} 
    A dual-convex 
    $m\times n$ net with pairwise-non-parallel faces $p_{kl}$ 
    is flexible in $I^3$ if and only if at least one of the following conditions holds:  
    \begin{enumerate}
    \item[\textup{(i)}] 
    $n$ lines $p_{k,0}\cap p_{k+2,0},\ldots,p_{k,n-1}\cap p_{k+2,n-1}$ lie in one isotropic plane for each $0\le k\le m-3$ or $m$ lines $p_{0,l}\cap p_{0,l+2},\ldots,p_{m-1,l}\cap p_{m-1,l+2}$ lie in one isotropic plane for each $0\le l\le n-3$;
    \item[\textup{(ii)}] any two non-boundary vertices joined by an edge 
    have equal opposite ratios with respect to it.
\end{enumerate}
\end{theorem}
\fixskip

\begin{figure}
    \centering
    \includegraphics[scale=0.07]{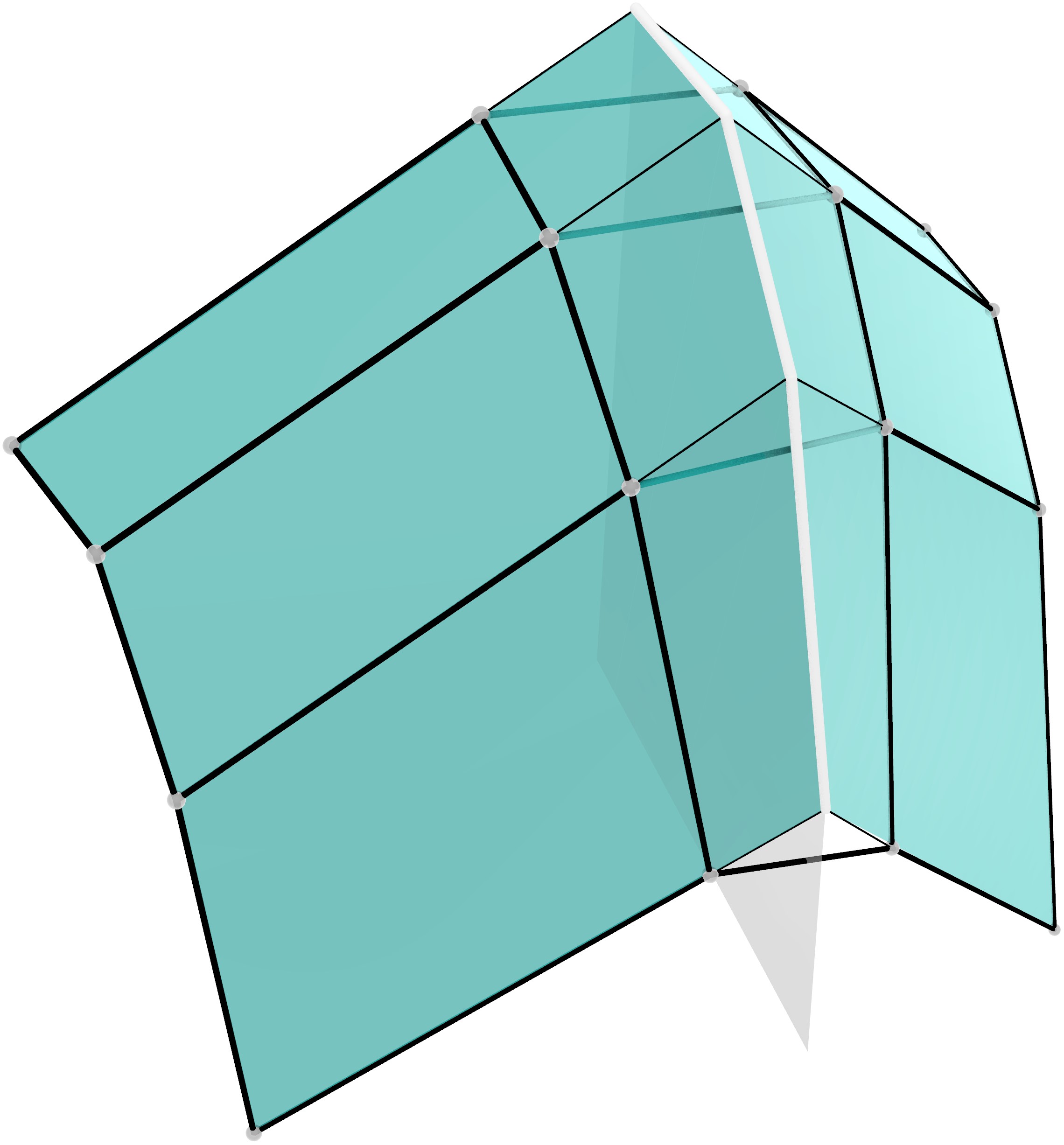}
    \caption{A flexible 
    net in $I^3$ of class~(i). 
    The intersections (white lines) of the planes of every other face lie in the same isotropic plane (gray). This condition implies that the extensions (thin lines) of every other edge on one family of parameter lines intersect or are parallel (thus the family of parameter lines is planar).
    Under some convexity assumptions, extending every other face until they meet, we construct a generalized T-net. The initial net is obtained from the resulting one by ``truncation of edges''.
    See Theorem~\ref{th-flexible}(i). }
    \label{fig:flexible-class-i-def}
\end{figure}

\begin{remark} Theorem~\ref{th-flexible} remains true if parallel faces are allowed in the dual-convex $m\times n$ net, but condition~(i) is replaced with the following one:
\begin{enumerate}
    \item[\textup{(i$'$)}] 
    for each $0\le k\le m-3$ the intersections $p_{k,0}\cap p_{k+2,0},\ldots,p_{k,n-1}\cap p_{k+2,n-1}$ are either all empty or all nonempty and lie in one isotropic plane or for each $0\le l\le n-3$ the intersections $p_{0,l}\cap p_{0,l+2},\ldots,p_{m-1,l}\cap p_{m-1,l+2}$ are either all empty or all nonempty and lie in one isotropic~plane.
\end{enumerate}
\end{remark}
\fixskip


Theorem~\ref{th-flexible} allows us to check if a given net is flexible. Now we 
parametrize flexible nets, allowing us to construct all flexible nets close enough to a certain standard one.

A \emph{wide L-shaped net of size $m\times n$} 
is an indexed collection of $3m+3n-3$ points $F_{ij}$, where the indices satisfy the inequalities $0\le i\le m\ge 2$, 
$0\le j\le n\ge 2$, 
$\min\{i,j\}\le 2$, such that $F_{ij},F_{i+1,j},F_{i,j+1},F_{i+1,j+1}$ are consecutive vertices of a convex quadrilateral whenever $i\in\{0,1\}$, $0\le j< n$ or $j\in\{0,1\}$, $0\le i< m$. See Figure~\ref{fig:wide-L-shaped-net}. 
\emph{Faces}, \emph{edges}, 
\emph{boundary vertices},  \emph{sub-nets}, and \emph{dual-convexity} are defined analogously to the ones of an $m\times n$ net. 
The wide L-shaped net 
with $F_{ij}=(i,j,i^2+j^2)$ 
is called \emph{standard}.
An $m\times n$ net \emph{extends} a wide L-shaped net of size $m\times n$, if the planes of faces $p_{ij}$ 
in the two nets coincide whenever $i\in\{0,1\}$, $0\le j< n$ or $j\in\{0,1\}$, $0\le i< m$.   

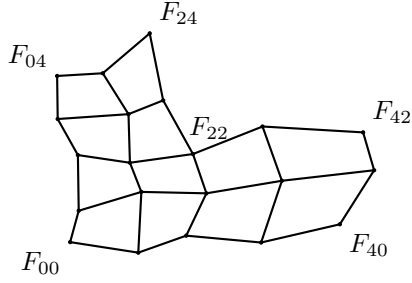
\begin{figure}
    \centering
    \begin{tikzpicture}[scale=0.2]

    \coordinate (f00) at (-6.514238984885258,-7.5200909923543255);
    \coordinate (f01) at (-5.950485913957726,-5.439573706790447);
    \coordinate (f02) at (-6.004176682617958,-1.7483333614345793);
    \coordinate (f03) at (-7.333023206945937,0.6409058439228951);
    \coordinate (f04) at (-7.370842141705287,3.4951126671782276);
    \coordinate (f10) at (-2.031059801798932,-8.218070984930659);
    \coordinate (f11) at (-1.8297194193244617,-4.1912633354519775);
    \coordinate (f12) at (-2.5813901805609314,-2.258395663701887);
    \coordinate (f13) at (-2.675349025715435,0.9630504558812242);
    \coordinate (f14) at (-4.366608238496497,3.6744342731970128);
    \coordinate (f20) at (1.1232728569599566,-7.0905648430768045);
    \coordinate (f21) at (2.4521193812879423,-4.285222180606488);
    \coordinate (f22) at (1.5662216984020454,-1.6812199006099324);
    \coordinate (f23) at (-0.3934913576774902,1.8757935230964415);
    \coordinate (f24) at (-1.2793890405629182,6.318704629688132);
    \coordinate (f30) at (6.049400881488481,-7.546936376684157);
    \coordinate (f31) at (7.418515482311733,-3.4530152663808447);
    \coordinate (f32) at (6.143359726642856,0.1442662338204962);
    \coordinate (f40) at (11.230560057151624,-6.338894081840734);
    \coordinate (f41) at (13.485572340859738,-2.7550352738045016);
    \coordinate (f42) at (12.760746963953556,-0.24499183896263965);

    \draw[black, thick] (f00) -- (f01) -- (f02) -- (f03) -- (f04) -- (f14) -- (f24) -- (f23) -- (f22) -- (f32) -- (f42) -- (f41) -- (f40) -- (f30) -- (f20) -- (f10) -- (f00);
    \draw[black, thick] (f01) -- (f11) -- (f21) -- (f31) -- (f41);
    \draw[black, thick] (f02) -- (f12) -- (f22);
    \draw[black, thick] (f03) -- (f13) -- (f23);
    \draw[black, thick] (f10) -- (f11) -- (f12) -- (f13) -- (f14);
    \draw[black, thick] (f20) -- (f21) -- (f22);
    \draw[black, thick] (f30) -- (f31) -- (f32);

    \draw (f04) node[anchor=south east] {${F_{04}}$};
    \draw (f00) node[anchor=north east] {${F_{00}}$};
    \draw (f42) node[anchor=south west] {${F_{42}}$};
    \draw (0.7,-1.3) node[anchor=south west] {${F_{22}}$};
    \draw (f40) node[anchor=north west] {$F_{40}$};
    \draw (f24) node[anchor=south west] {$F_{24}$};

    \draw [fill=black] (-6.004176682617958,-1.7483333614345793) circle (3pt);
    \draw [fill=black] (-7.333023206945937,0.6409058439228951) circle (3pt);
    \draw [fill=black] (-2.675349025715435,0.9630504558812242) circle (3pt);
    \draw [fill=black] (-2.5813901805609314,-2.258395663701887) circle (3pt);
    \draw [fill=black] (-1.8297194193244617,-4.1912633354519775) circle (3pt);
    \draw [fill=black] (-5.950485913957726,-5.439573706790447) circle (3pt);
    \draw [fill=black] (1.5662216984020454,-1.6812199006099324) circle (3pt);
    \draw [fill=black] (2.4521193812879423,-4.285222180606488) circle (3pt);
    \draw [fill=black] (-6.514238984885258,-7.5200909923543255) circle (3pt);
    \draw [fill=black] (-2.031059801798932,-8.218070984930659) circle (3pt);
    \draw [fill=black] (1.1232728569599566,-7.0905648430768045) circle (3pt);
    \draw [fill=black] (6.143359726642856,0.1442662338204962) circle (3pt);
    \draw [fill=black] (7.418515482311733,-3.4530152663808447) circle (3pt);
    \draw [fill=black] (6.049400881488481,-7.546936376684157) circle (3pt);
    \draw [fill=black] (-7.370842141705287,3.4951126671782276) circle (3pt);
    \draw [fill=black] (-4.366608238496497,3.6744342731970128) circle (3pt);
    \draw [fill=black] (-0.3934913576774902,1.8757935230964415) circle (3pt);
    \draw [fill=black] (-1.2793890405629182,6.318704629688132) circle (3pt);
    \draw [fill=black] (12.760746963953556,-0.24499183896263965) circle (3pt);
    \draw [fill=black] (13.485572340859738,-2.7550352738045016) circle (3pt);
    \draw [fill=black] (11.230560057151624,-6.338894081840734) circle (3pt);

\end{tikzpicture}
    \caption{A wide L-shaped net of size $4\times4$. See Corollary~\ref{cor-L}.} 
    \label{fig:wide-L-shaped-net}
\end{figure}


\fixskip
\begin{corollary}  \label{cor-L}
If a dual-convex wide L-shaped net of size $m\times n$ is sufficiently close to the standard one and satisfies one of the following conditions, 
then 
there exists a unique flexible dual-convex $m\times n$ net in $I^3$ that extends it: 
\begin{enumerate}
\item[\textup{(i${}''$)}] 
     the pair of lines $p_{k,0}\cap p_{k+2,0}$ and $p_{k,1}\cap p_{k+2,1}$ lies in one isotropic plane for each $0\le k\le m-3$ or the pair of lines $p_{0,l}\cap p_{0,l+2}$ and $p_{1,l}\cap p_{1,l+2}$ lies in one isotropic plane for each $0\le l\le n-3$.
\item[\textup{(ii)}] any two non-boundary vertices joined by an edge 
    have equal opposite ratios with respect to it.
\end{enumerate} 

\end{corollary}
\fixskip

As we shall see, the classification of flexible  $m\times n$ nets in $I^3$ reduces to the classification of deformable $(m-1)\times(n-1)$ nets, which was obtained in~\cite{pirahmad2024area} and recalled in Subsection~\ref{ssec-classification-deformable} below.

\fixskip
\begin{definition} \label{def-deformable-net} 
(See \cite[Definition~1]{pirahmad2024area}.)
An $m\times n$ net is \emph{deformable} if contained in a continuous family of pairwise non-congruent (in Euclidean or isotropic sense) parallel $m\times n$ nets 
with the same areas of corresponding faces. 
\end{definition}
\fixskip



We have the following analog of Proposition~\ref{prop-infinitesimal-characterization}. The reciprocal-parallel net is defined in Section~\ref{ssec-infinitesimal-statements} above. A \emph{dual-affine transformation} is a projective transformation of (the projectivization of) $I^3$ keeping the $z$-direction. The metric duality takes dual-affine transformations to affine ones and vice versa.

\fixskip
\begin{prop} \label{l-flexible-vs-deformable}
For a dual-convex $m\times n$ net, 
the following conditions are equivalent:
\begin{description}
    \item[\textup{\textbf{(I)}}] the $m\times n$ net is flexible in $I^3$;
    \item[\textup{\textbf{(D)}}] the metric dual $(m-1)\times (n-1)$ net is deformable;
    \item[\textup{\textbf{(R)}}] the $m\times n$ net has a reciprocal-parallel net with a deformable top view.  
\end{description}
\end{prop}
\fixskip

\begin{corollary} If a dual-convex $m\times n$ net is flexible in $I^3$, then all its Combescure and 
dual-affine transformations
are.
\end{corollary}
\fixskip

The reader interested in the proofs of  Theorem~\ref{th-flexible}, Corollary~\ref{cor-L}, and Proposition~\ref{l-flexible-vs-deformable} can proceed to Section~\ref{ssec-classification-deformable}, 
and now we discuss geometric properties of the resulting classes~(i) and (ii) of flexible $m\times n$ nets in $I^3$.

\subsection{Geometric properties}
\label{ssec:Finite:Geometric properties}

\subsubsection{Class (i)}


Class (i) can be described in terms of dual-affine transformations. Two non-boundary vertices joined by an edge in a dual-convex $m\times n$ net are called \emph{dual-affine-symmetric with respect to the edge} if there is a dual-affine transformation taking the planes of the faces around the first vertex to the ones around the second vertex, and keeping two of those planes containing the edge fixed. By the metric duality and Proposition~\ref{th-old} below, 
condition~(i) in Theorem~\ref{th-flexible} is equivalent to the following one: \emph{on each parameter line of one family, any two adjacent non-boundary vertices are dual-affine-symmetric with respect to the edge between them}.

Any net of class~(i) has one family of planar parameter lines as explained in the caption of Figure~\ref{fig:flexible-class-i-def}.
The other family of parameter lines is not planar in general. However, if we intersect the planes of every other face, we get polylines in isotropic planes as shown in Figure~\ref{fig:flexible-class-i-def}. Note that the class of planar-quad nets whose parameter lines lie in planes is invariant under Combescure transformations and so is the property of parameter lines lying in isotropic planes.  

Within class~(i), there are special nets with particularly nice properties: \emph{the dual-convex $m\times n$ nets with planar parameter lines, where one family of parameter lines lie in isotropic planes}.
We call them \emph{generalized T-nets}. Any net of class~(i), subject to minor convexity assumptions, is obtained from some generalized T-net by ``truncation'' of edges (there are actually two such generalized T-nets). See Figure~\ref{fig:flexible-class-i-def}. 

Generalized T-nets can be viewed as the isotropic counterparts of the T-nets, a known class of flexible nets in Euclidean geometry \cite{sauer:1970}. Indeed, T-nets can be characterized by the property that the parameter lines are planar, and the planes of any two parameter lines from distinct families are orthogonal~\cite[Definition~2.1]{izmestiev2024isometric}. In isotropic geometry, the orthogonality of two families of planes means that one of the families consists of isotropic planes, and we come to the notion of generalized T-nets. For any T-net, the planes of one family of parameter lines are parallel~\cite[Lemma~2.2]{izmestiev2024isometric}. By a (Euclidean) rotation, one can make those planes parallel or orthogonal to the $z$-axis; in the latter case, the other family of planes is parallel to the $z$-axis. In either case, we get a generalized T-net, flexible in both Euclidean and isotropic geometries. But the class of generalized T-nets is much larger: in particular, their faces are not necessarily trapezoids. Class~(i) is even larger.




As we shall see, generalized T-nets are metric dual to \emph{cone-cylinder nets} given by the equation \cite[\S3.2.1]{pirahmad2024area}
\begin{equation}
    \label{eq-p-cone-cylinder-nets}
    P_{ij}=a_i + \sigma_i b_j,
    \qquad
    0 \leq i \leq m+1\text{ and } 0 \leq j \leq n+1,
\end{equation}
for some $a_0,\dots,a_{m+1},b_0,\dots,b_{n+1}\in\mathbb{R}^3$ and $\sigma_0,\dots,\sigma_{m+1}\in\mathbb{R}$ (up to interchanging the indices $i$ and~$j$). 
Cone-cylinder nets are a particular case of double cone-nets studied in \cite{kilian2023smooth}.
Cone-cylinder nets admit an explicit formula for the deformation.

\fixskip
\begin{prop}\textup{(See \cite[Proposition~12]{pirahmad2024area}.)}
\label{p-cone-cylinder-deformation}
    A cone-cylinder net~\eqref{eq-p-cone-cylinder-nets} with all $\sigma_i>0$ is embedded into a family 
\begin{equation} 
\label{eq:discrete-family}
    P_{ij}(t):=
    a_0+\sum_{k=1}^{i} \frac{\left(a_k-a_{k-1}\right)(\sigma_{k}+\sigma_{k-1})}
    {\sqrt{t+\sigma_{k}^2\vphantom{t+\sigma_{k-1}^2}}+\sqrt{t+\sigma_{k-1}^2}} 
    +
    \sqrt{t+\sigma_{i}^2}\,
    b_j,\ 
    0 \leq i \leq m+1\text,\, 0 \leq j \leq n+1,
\end{equation} 
    of cone-cylinder nets, which are 
    its area-preserving Combescure transformations, for some $\varepsilon>0$ and all $t\in [0,\varepsilon]$. 
\end{prop}
\fixskip

Applying the metric duality, 
we get an explicit isotropic isometric deformation of a generalized T-net. Recall that $e_1=(1,0,0)^\mathrm{T}$, $e_2=(0,1,0)^\mathrm{T}$, $ e_3=(0,0,1)^\mathrm{T}$.

\fixskip
\begin{prop}\label{eq-generalized-T-net-flexion}
Any $m\times n$ generalized T-net 
has the form (up to interchanging the indices $i$ and $j$) 
$$\small F_{ij} = - \frac{1}{\det(e_3, b_{j+1}-b_j,\Delta_{ij})}
\begin{pmatrix}
    \det(e_1, b_{j+1}-b_j,\Delta_{ij}) \\
    \det(e_2, b_{j+1}-b_j,\Delta_{ij}) \\
    \det(a_i+\sigma_ib_j, b_{j+1}-b_j,\Delta_{ij})
\end{pmatrix},
\ \ 0\leq i\leq m, 0\leq j\leq n,
$$
for some $a_0,\dots,a_{m+1},b_0,\dots,b_{n+1}\in\mathbb{R}^3$ and $\sigma_0,\dots,\sigma_{m+1}>0$, where 
$$
\Delta_{ij}:=a_{i+1}-a_i+b_j(\sigma_{i+1}-\sigma_{i})
$$
and 
$\det(e_3, b_{j+1}-b_j,\Delta_{ij})\ne 0$ for all $0\leq i\leq m$, $0\leq j\leq n$. 

It has a nontrivial isotropic isometric deformation
\begin{equation}
\label{eq:is-is-def}
    \small F_{ij}(t) = - \frac{1}{\det(e_3, b_{j+1}-b_j,\Delta_{ij})}
    \begin{pmatrix}
        \det(e_1, b_{j+1}-b_j,\Delta_{ij}) \\
        \det(e_2, b_{j+1}-b_j,\Delta_{ij}) \\
        \det(P_{ij}(t), b_{j+1}-b_j,\Delta_{ij})
    \end{pmatrix},
    \quad 0\leq i\leq m, 0\leq j\leq n
\end{equation}
for some $\varepsilon>0$ and all $t\in[0,\varepsilon]$,
where $P_{ij}(t)$ is defined by~\eqref{eq:discrete-family}. 
\end{prop}
\fixskip

We prove this proposition in Section~\ref{ssec-proof-flexible}.


 Another interesting class of flexible nets in $I^3$ is 
 dual-convex $m\times n$ nets with both families of parameter lines contained in isotropic planes. They are metric dual to translational nets and are isotropic analogs of Voss nets; see \cite{isometric-isotropic}. Indeed, Voss nets are characterized by the condition that the opposite face angles around each vertex are equal. In isotropic geometry, this means that the edges around an ``admissible'' vertex lie in two isotropic planes because the isotropic angles are measured in the top view. Thus, the isotropic analogs of Voss nets are particular cases of generalized T-nets, in contrast to Euclidean Voss nets.

\subsubsection{Class (ii)}

%

For nets of class~(ii), we have the following analog of Proposition~\ref{prop-infinitesimal-characterization}\textbf{(V)}.

\fixskip
\begin{prop}\label{p-class-ii-v-parallel} A dual-convex 
$m\times n$ net satisfies condition~(ii) of Theorem~\ref{th-flexible} if and only if it has a v-parallel one with opposite curvatures and vanishing mixed curvature of the corresponding vertices.
\end{prop}
\fixskip

We prove Proposition~\ref{p-class-ii-v-parallel} in  Section~\ref{ssec-proof-flexible}.

Here the conditions on the curvatures are symmetric, hence the v-parallel $m\times n$ net also belongs to class (ii) and is flexible in $I^3$. Such pairs of nets can be viewed as discrete relative minimal surfaces with relative Gaussian curvature $-1$ in the sense of \cite[Eq.~(22)]{pottmann2007discrete} (where we remove the assumption that one of the nets approximates the isotropic unit sphere, meaning passage to the \emph{relative} curvature). 

Dual-convex $m\times n$ nets with both families of parameter lines contained in isotropic planes 
belong to both classes~(i) and~(ii). For the other nets in class~(ii), we observe the same \emph{zig-zag phenomenon} as for deformable nets \cite[\S3.2.2]{pirahmad2024area}: the top views of the parameter lines have a ``zig-zag'' shape. A similar phenomenon is known in Euclidean geometry, e.g. for Miura-ori \cite[Figure~12]{izmestiev2024isometric}.


\subsection{An auxiliary classification}
\label{ssec-classification-deformable}

Theorem~\ref{th-flexible} will be obtained by the metric duality from the known classification of deformable $m\times n$ nets from \cite{pirahmad2024area}, which we recall now. 

We need the following notions. A pair of quadrilaterals in $I^3$ with a common side is called \emph{affine symmetric with respect to the common side} if there is an affine map taking the first quadrilateral to the second one and keeping the points of the common side fixed. 

The ratio of the areas of triangles $AQB$ and $CQD$ in a quadrilateral $ABCD$ with the diagonals meeting at $Q$ is called the \emph{opposite ratio of 
$ABCD$ with respect to the side} $AB$. It equals 
$AQ\cdot BQ/(CQ\cdot DQ)$ and also $\|AQ\|_i\cdot \|BQ\|_i/(\|CQ\|_i\cdot \|DQ\|_i)$ if the plane of the quadrilateral is non-isotropic. Notice that if two quadrilaterals are affine symmetric with respect to the common side, then their opposite ratios with respect to all corresponding sides are equal. 

\fixskip
\begin{theorem} \textup{(See \cite[Theorem~9]{pirahmad2024area}.)}
\label{th-mxn}
An $m\times n$ net is deformable if and only if one of the following conditions~holds:
\begin{itemize}
\item[\textup{(i$^*$)}] in each $1\times 2$ sub-net or in each $2\times 1$ sub-net, the two faces are affine symmetric with respect to their common edge;
\item[\textup{(ii$^*$)}] each pair of faces with a common edge has equal opposite ratios with respect to that edge.
\end{itemize}
\end{theorem}
\fixskip

We are going to use the following equivalent forms of conditions~(i$^*$) and~(ii$^*$). 

\fixskip
\begin{prop}\label{th-old} \textup{(See \cite[Proposition~2]{pirahmad2024area} and Figure~\ref{figure:properties}.)}
Two non-coplanar convex quadrilaterals $ABCD$ and $ABC'D'$ in $\mathbb{R}^d$ 
are affine symmetric with respect to their common side $AB$, if and only if $CC'\parallel DD'$.
%
%
%
%
\end{prop}
\fixskip

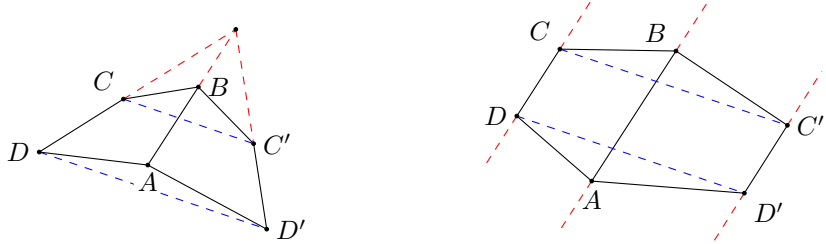
\begin{figure}[htbp]
    \centering
    \begin{tikzpicture}[scale=0.17]
        \coordinate (D) at (-2.50143, 1.26455);
        \coordinate (C) at (4.02081, 5.3953);
        \coordinate (B) at (9.81836, 6.3374);
        \coordinate (C') at (14.09405, 1.91677);
        \coordinate (D') at (15.10862, -4.75041);
        \coordinate (A) at (3.29612, -3.80831);
        \coordinate (P) at (12.71713, 10.8305);
        \coordinate (A') at ($(A)!0.4!(B)$);
    
        \draw[black, thin] (D) -- (C) -- (B) -- (C') -- (D') -- (A') -- cycle;
        \draw[black, thin] (A') -- (B);
        
        \draw[blue, dashed] (D) -- (D');
        \draw[blue, dashed] (C) -- (C');
        \draw[red, dashed] (C) -- (P) -- (B);
        \draw[red, dashed] (P) -- (C');
        
        \filldraw[black] (D) circle (4pt) node[anchor=east]{$D$};
        \filldraw[black] (C) circle (4pt) node[anchor=south east]{$C$};
        \filldraw[black] (B) circle (4pt) node[anchor=west]{$B$};
        \filldraw[black] (C') circle (4pt) node[anchor=west]{$C'$};
        \filldraw[black] (D') circle (4pt) node[anchor=west]{$D'$};
        \node[draw=white, fill=white, circle, inner sep=0.005pt, yshift=-6.5pt] at (A') {$A$};
        \filldraw[black] (A') circle (4pt);
        \filldraw[black] (P) circle (4pt) node[anchor=west]{};
    
    \end{tikzpicture}\qquad\qquad\qquad 
    \begin{tikzpicture}[scale=0.17]
        \coordinate (D) at (-2.50143, 1.26455);
        \coordinate (C) at (0.83216, 6.48234);
        \coordinate (B) at (9.81836, 6.3374);
        \coordinate (C') at (18.44221, 0.53985);
        \coordinate (D') at (15.10862, -4.75041);
        \coordinate (A) at (3.29612, -3.80831);
        \coordinate (C+) at ($(D)!1.8!(C)$);
        \coordinate (B+) at ($(A)!1.4!(B)$);
        \coordinate (C'+) at ($(D')!1.8!(C')$);
        \coordinate (D+) at ($(C)!1.8!(D)$);
        \coordinate (A+) at ($(B)!1.4!(A)$);
        \coordinate (D'+) at ($(C')!1.8!(D')$);

        \draw[red, dashed] (D) -- (D+);
        \draw[red, dashed] (A) -- (A+);
        \draw[red, dashed] (D') -- (D'+);
        \draw[black, thin] (C) -- (B) -- (C'); 
        \draw[black, thin] (D) -- (A) -- (D');
        \draw[black, thin] (D) -- (C);
        \draw[black, thin] (C') -- (D');
        \draw[black, thin] (B) -- (A);
        
        \draw[red, dashed] (C') -- (C'+);
        \draw[red, dashed] (B) -- (B+);
        \draw[red, dashed] (C) -- (C+);
        \draw[blue, dashed] (D) -- (D');
        \draw[blue, dashed] (C) -- (C');
        
        \filldraw[black] (D) circle (4pt) node[anchor=east]{$D$};
        \filldraw[black] (C) circle (4pt) node[anchor=south east]{$C$};
        \filldraw[black] (B) circle (4pt) node[anchor=south east]{$B$};
        \filldraw[black] (C') circle (4pt) node[anchor=west]{$C'$};
        \filldraw[black] (D') circle (4pt) node[anchor=north west]{$D'$};
        \filldraw[black] (A) circle (4pt) node[anchor=north]{$A$};
    \end{tikzpicture}
    \caption{The characteristic property of two non-coplanar affine symmetric quadrilaterals: the blue lines are parallel. Then the red lines are either concurrent or parallel. See Proposition~\ref{th-old} and \cite[Proposition~2]{pirahmad2024area}.}
    \label{figure:properties}
\end{figure}

\begin{prop}
    \label{th-christoffel} \textup{(See \cite[Proposition~13]{pirahmad2024area}.)}
    An $m\times n$ net satisfies condition~\textup{(ii$^*$)} of Theorem~\ref{th-mxn} if and only if it has a Christoffel dual with the same (unoriented) areas of corresponding faces. 
\end{prop}
\fixskip

\begin{corollary} \label{cor-deformable-Christoffel-dual} \textup{(See \cite[Corollary~4]{pirahmad2024area}.)}
A deformable $m\times n$ net has a deformable Christoffel dual.
\end{corollary}

\fixskip
Within class~(i$^*$), there is a sub-class of cone-cylinder nets~\eqref{eq-p-cone-cylinder-nets} with particularly nice properties. They are characterized geometrically as follows.

\fixskip
\begin{prop}\textup{(See \cite[Proposition~11]{pirahmad2024area}.)} \label{p-cone-cylinder-nets}
An $m\times n$ net has form 
$P_{ij}=a_i+\sigma_i b_j$ for all $0 \leq i \leq m$, $0 \leq j \leq n$, 
if and only if 
the $m+1$ lines $P_{0,j}P_{0,j+1},\ldots, P_{m,j}P_{m,j+1}$ are parallel for each $0\leq j<n$ and the $n+1$ lines $P_{i,0}P_{i+1,0},\ldots, P_{i,n}P_{i+1,n}$ are either parallel or concurrent for each $0\leq i<m$. 
\end{prop} 
\fixskip

For the proof of Corollary~\ref{cor-L}, we need the following notions and results.
By an \emph{L-shaped net of size $m\times n$} we mean an indexed collection of $2m+2n$ points $P_{ij}$, where the indices satisfy the inequalities $0\le i\le m$, $0\le j\le n$, $\min\{i,j\}\le 1$, such that $P_{ij},P_{i+1,j},P_{i,j+1},P_{i+1,j+1}$ are consecutive vertices of a convex quadrilateral whenever $i=0$, $0\le j< n$ or $j=0$, $0\le i< m$. 
An \emph{L-shaped square net} is an L-shaped net such that \bluenew{all} the faces are coplanar non-coincident squares.

\fixskip
\begin{corollary}\textup{(See \cite[Corollary~3]{pirahmad2024area}.)}  \label{cor-LL}
    If an L-shaped net of size $m\times n$ is sufficiently close to an L-shaped square net and satisfies one of 
    conditions (i$^*$) or (ii$^*$) in Theorem~\ref{th-mxn}, then it is contained in exactly one deformable  $m\times n$ net. 
\end{corollary}

\subsection{Proof of the classification}
\label{ssec-proof-flexible}

In this subsection, we prove Propositions~\ref{l-flexible-vs-deformable}, \ref{eq-generalized-T-net-flexion}, \ref{p-class-ii-v-parallel}, Theorem~\ref{th-flexible}, and Corollary~\ref{cor-L}. 
For the former, we need two lemmas.


\fixskip
\begin{lemma} \label{l-isotropic-flexible} 
A dual-convex $m\times n$ net is flexible in $I^3$ if and only if it has a nontrivial isotropic isometric deformation consisting of v-parallel $m\times n$ nets.
\end{lemma}

\begin{proof} The `if' part follows from the definition of a flexible net in $I^3$. Let us prove the `only if' part. Let $F_{ij}(t)$, where $0\leq i\leq m,0\leq j\leq n$, be a nontrivial isotropic isometric deformation of a dual-convex $m\times n$ net $F_{ij}$.
Then the top view 
$\overline{F_{ij}(t)}$ 
is a Euclidean isometric deformation of $\overline{F_{ij}}$.
Clearly, any $m\times n$ net in the plane is rigid, i.e., all planar Euclidean isometric deformations are trivial. (Indeed, 
let $C_{kl}(t)$ be the unique planar Euclidean congruence taking 
a face $p_{kl}(t)$ of the Euclidean isometric deformation $\overline{F_{ij}(t)}$ to the face $p_{kl}$ of the net $\overline{F_{ij}}$ and preserving the indices of the vertices. Since $\overline{F_{ij}(t)}$ depends on $t$ continuously, $C_{kl}(t)$ does, hence $C_{kl}(t)$ preserves orientation. The images of the common vertices of neighboring faces coincide under the corresponding maps, hence the orientation-preserving congruences $C_{kl}(t)$ for neighboring faces coincide as well. Thus all $C_{kl}(t)$ are equal and the deformation is trivial.)
Thus 
there is a Euclidean isometric deformation $C_t\colon \mathbb{R}^2\to \mathbb{R}^2$ such that 
$\overline{F_{ij}(t)}=C_t(\overline{F_{ij}})$ for all $0\leq i\leq m,0\leq j\leq n$.
Extend $C_t$ to an isotropic isometric deformation of $I^3$ preserving the $z$-coordinate; it is still denoted by $C_t$. Then $C_t^{-1}(F_{ij}(t))$ is the desired nontrivial isotropic isometric deformation consisting of v-parallel $m\times n$ nets. 
\end{proof}


\begin{lemma} \label{lem-col-L}
    There exists a unique $m\times n$ net 
    with a given pair of parameter lines from distinct families and a given planar $m\times n$ net in the top view.
\end{lemma}

\begin{proof}
      Without loss of generality, assume that the two parameter lines contain vertices \(F_{ij}\) with \(i=0\) or \(j=0\).
      We then determine the remaining vertices 
      inductively. 
      Suppose that the points \(F_{i,j+1}\), \(F_{ij}\), \(F_{i+1,j}\) have already been determined. Then the point \(F_{i+1,j+1}\) is uniquely determined because its top view is given and it lies in the plane through 
      \(F_{i,j+1}\), \(F_{ij}\), \(F_{i+1,j}\). The resulting face is 
      convex because the top view is.
\end{proof}

\begin{lemma} \label{l-deformable-vs-top-view}
    Let the planes of the faces of an $m\times n$ net be non-isotropic. Then the $m\times n$ net is deformable if and only if 
    its top view is.
\end{lemma}

\begin{proof}
      Assume that an $m\times n$ net $F_{ij}$ is deformable, and $F_{ij}(t)$ is a family of its noncongruent area-preserving Combescure transformations. Then $\overline{F_{ij}(t)}$ is a family of noncongruent area-preserving Combescure transformations of the top view $\overline{F_{ij}}$. 

      Conversely, if the top view $\overline{F_{ij}}$ is deformable 
      and $\overline{F_{ij}}(t)$ is a family of its noncongruent area-preserving Combescure transformations, then we may assume that $\overline{F_{00}}(t)=\overline{F_{00}}$. Set
      $F_{00}(t)=F_{00}$ and construct $F_{i0}(t)$ for $i=1,\dots,m$ inductively as the unique point with the top view $\overline{F_{i0}}(t)$ on the line passing through $F_{i-1,0}(t)$ parallel to $F_{i-1,0}F_{i0}$. Analogously, construct $F_{0j}(t)$ for $j=1,\dots,n$. Lift $\overline{F_{ij}}(t)$ to a net ${F_{ij}(t)}$ using Lemma~\ref{lem-col-L}. By induction, ${F_{ij}(t)}$ is a family of noncongruent area-preserving Combescure transformations of the $m\times n$ net ${F_{ij}}$.
\end{proof}


\begin{proof}[Proof of Proposition~\ref{l-flexible-vs-deformable}]
    Let us prove the equivalence \textbf{(I)}$\Leftrightarrow$\textbf{(D)}.
    By Lemma~\ref{l-isotropic-flexible}, condition~\textbf{(I)} is equivalent to having a nontrivial isotropic isometric deformation consisting of
    v-parallel $m\times n$ nets. The face condition holds automatically for such an isotropic isometric deformation, and the metric dual nets are parallel. The vertex condition
    is equivalent to the condition that the corresponding faces of the metric dual nets have the same areas, hence to condition~\textbf{(D)}.
    We get \textbf{(I)}$\Leftrightarrow$\textbf{(D)}.
    
    Let us prove the equivalence \textbf{(I)}$\Leftrightarrow$\textbf{(R)}.
    Clearly, a flexible dual-convex $m\times n$ net in $I^3$ is also infinitesimally flexible in $I^3$ (this follows, e.g., from \textbf{(I)}$\Leftrightarrow$\textbf{(D)}, Proposition~\ref{prop-infinitesimal-characterization}\textbf{(I,K)}, and Corollary~\ref{cor-deformable-Christoffel-dual}). Then, by Proposition~\ref{prop-infinitesimal-characterization} it has 
    a reciprocal-parallel net. By Corollary~\ref{p-reciprocal-parallel-top-view} the top views of the reciprocal-parallel and the metric dual nets are related by a rotation through $\pi/2$ followed by the Christoffel duality. By the equivalence \textbf{(I)}$\Leftrightarrow$\textbf{(D)}, the initial $m\times n$ net is flexible if and only if the
    metric dual, or, equivalently (by Lemma~\ref{l-deformable-vs-top-view}), its top view, is deformable. By Corollary~\ref{cor-deformable-Christoffel-dual}, the rotation and the Christoffel duality preserve the deformability, and the result follows.
\end{proof}



\begin{proof}[Proof of Theorem~\ref{th-flexible}]
    According to Proposition~\ref{l-flexible-vs-deformable},  a dual-convex $m\times n$ net with faces $p_{ij}$ is flexible in \( I^3 \) if and only if the metric dual \((m-1) \times (n-1)\) net \( p_{ij}^* \) is deformable. The latter is equivalent to one of conditions
    (i$^*$) or (ii$^*$) of Theorem~\ref{th-mxn}. Consider the two possibilities separately. Notice that no two neighboring faces of the metric dual net \( p_{ij}^* \) are coplanar, otherwise 
    their metric dual neighboring vertices of the original $m\times n$ net $p_{ij}$ coincide, contradicting the definition of an $m\times n$ net.
    
    \emph{Case} (i$^*$): \emph{the net \( p_{ij}^* \) satisfies condition}~(i$^*$). By Proposition~\ref{th-old}, this is equivalent to \( p_{k,0}^*p_{k+2,0}^* \parallel p_{k,1}^*p_{k+2,1}^* \parallel \ldots \parallel p_{k,n-1}^*p_{k+2,n-1}^* \) for each \( 0 \leq k \leq m-3 \), or \( p_{0,l}^*p_{0,l+2}^* \parallel p_{1,l}^*p_{1,l+2}^* \parallel \ldots \parallel p_{m-1,l}^*p_{m-1,l+2}^* \) for each \( 0 \leq l \leq n-3 \), because neighboring faces of the net \( p_{ij}^* \) are not coplanar. Several non-isotropic lines are parallel if and only if their metric duals lie in one isotropic plane because the metric dual of their intersection point (at infinity) is an isotropic plane containing the metric duals of the lines. We arrive at equivalent condition~(i).

     \emph{Case} (ii$^*$): \emph{the net \( p_{ij}^* \) satisfies condition}~(ii$^*$). This means that any two neighboring faces have equal opposite ratios with respect to the common side. The opposite ratio of a face with respect to an edge equals the opposite ratio of the metric dual vertex with respect to the metric dual edge because the metric duality takes isotropic distances to isotropic angles and reverses inclusions. We arrive at equivalent condition~(ii). 

    Hence, both cases indicate that the net is flexible if and only if at least one of the conditions (i) or (ii) of the theorem holds.
    \end{proof}

\begin{proof}[Proof of Corollary~\ref{cor-L}]
    Consider the metric dual of the wide L-shaped net. 
    It is an $(m-1)\times(n-1)$ L-shaped net 
    sufficiently close to the L-shaped net with the vertices $(2i+1,2j+1,i^2+j^2+i+j)$, where $i\in\{0,1\}$, $0\le j< n$ or $j\in\{0,1\}$, $0\le i< m$. The top view of the former is sufficiently close to the top view of the latter, which has square faces. By 
    Corollary~\ref{cor-LL} and Proposition~\ref{th-old}, the top view of the former is contained in a unique deformable $(m-1)\times(n-1)$ net. By Lemma~\ref{lem-col-L}, that deformable net lifts to a unique $(m-1)\times(n-1)$ net containing the 
    L-shaped net. The lifted net is deformable if and only if its top view is (Lemma~\ref{l-deformable-vs-top-view}). By Proposition~\ref{l-flexible-vs-deformable}, 
    we are done.
\end{proof}

\begin{proof}[Proof of Proposition~\ref{eq-generalized-T-net-flexion}]
    %
    %
    %
    An \(m\times n\) generalized T-net \(F_{ij}\) (with \(0 \leq i \leq m\), \(0 \leq j \leq n\)) is metric dual to some \((m+1)\times (n+1)\) cone-cylinder net \(P_{ij}\) (where \(0 \leq i \leq m+1\) and \(0 \leq j \leq n+1\)) by Proposition~\ref{p-cone-cylinder-nets} because a parameter line lying in an isotropic plane corresponds to parallel edges, and a planar parameter line corresponds to edges contained in either parallel or concurrent lines. Here all the non-boundary vertices 
    of the net \(P_{ij}\) are determined by the net \(F_{ij}\) so that the non-boundary faces are convex (by Lemma~\ref{l-dual-convex}), and the boundary vertices can easily be chosen so that the boundary faces are also convex (by taking $\sigma_0$ and $\sigma_{m+1}$ in~\eqref{eq-p-cone-cylinder-nets} close enough to $\sigma_1$ and $\sigma_{m}$ respectively). 
    Since the resulting net \(P_{ij}\) has convex faces, it follows that
    $\sigma_0,\dots,\sigma_{m+1}$ in~\eqref{eq-p-cone-cylinder-nets} have the same sign. Assume that all $\sigma_i>0$, otherwise reverse the signs of all $\sigma_i$ and $b_j$.
        
    Thus, it suffices to find the metric dual \(F_{ij}\) of each face of the cone-cylinder net \(P_{ij}\). The plane containing the points $P_{ij}, P_{i+1,j}, P_{i+1,j+1}, P_{i,j+1}$ is given by 
    \[
        z \;=\;
        -\,\frac{N \cdot e_1}{N \cdot e_3}\,x
        \;-\;
        \frac{N \cdot e_2}{N \cdot e_3}\,y
        \;+\;
        \frac{N \cdot P_{ij}}{N \cdot e_3},
    \]
    where $N \;=\; \bigl(P_{i,j+1} - P_{ij}\bigr)\,\times\,\bigl(P_{i+1,j} - P_{ij}\bigr)$ is the normal vector and we use the Euclidean dot and cross products. By the identity $(a \times b)\cdot c \;=\; \det(a,b,c)$, the metric dual of this plane is the point
    \begin{multline*}
        F_{ij} 
        =-\frac{1}{
          \det\bigl(e_3,\; P_{i,j+1}-P_{ij},\; P_{i+1,j}-P_{ij}\bigr)
        }
        \begin{pmatrix}
          \det\!\bigl(e_1,\; P_{i,j+1}-P_{ij},\; P_{i+1,j}-P_{ij}\bigr)\\[6pt]
          \det\!\bigl(e_2,\; P_{i,j+1}-P_{ij},\; P_{i+1,j}-P_{ij}\bigr)\\[6pt]
          \det\!\bigl(P_{ij},\; P_{i,j+1}-P_{ij},\; P_{i+1,j}-P_{ij}\bigr)
        \end{pmatrix}.
    \end{multline*}
    Substituting $P_{ij}$ from~(\ref{eq-p-cone-cylinder-nets}) 
    and canceling common factors $\sigma_i\ne 0$ in the numerators and denominators yields the desired 
    formula. 

    Finally, we show that $F_{ij}(t)$ defined in~(\ref{eq:is-is-def}) is a non-trivial isotropic isometric deformation. 
    It is straightforward that $F_{ij}(0) = F_{ij}$ and 
    the corresponding points $F_{ij}(t)$ have the same top view, implying the face condition. By~\eqref{eq-def-curvature}, their isotropic Gaussian curvatures of corresponding vertices coincide as well, because 
    \begin{multline*}
        \Omega(F_{ij}(t)):= \mathrm{Area}\left(\overline{P_{ij}(t)}\,\overline{P_{i+1,j}(t)}\,\overline{P_{i+1,j+1}(t)}\,\overline{P_{i,j+1}(t)}\right)=\\
        =\mathrm{Area}\left(\overline{P_{ij}}\,\overline{P_{i+1,j}}\,\overline{P_{i+1,j+1}}\,\overline{P_{i,j+1}}\right)=\Omega(F_{ij}).
    \end{multline*}
    Here the areas are equal because $P_{ij}(t)$ is area-preserving by Proposition~\ref{p-cone-cylinder-deformation}. The isotropic isometric deformation $F_{ij}(t)$ is nontrivial because $P_{ij}(t)$ are not translations of each other.
    %
\end{proof}

\begin{proof}[Proof of Proposition~\ref{p-class-ii-v-parallel}]
By Lemmas~\ref{l-mixed-deformable} and~\ref{l-infinitesimal-deformable},
the metric duality takes a pair of v-parallel dual-convex $m\times n$ nets with vanishing mixed curvature of the corresponding vertices to a pair of Clifford dual $(m-1)\times (n-1)$ nets.
The $m\times n$ nets have in addition opposite curvatures at the corresponding vertices if and only if the $(m-1)\times (n-1)$ nets have the same (unoriented) areas of the corresponding faces (their oriented areas have opposite signs automatically). By Proposition~\ref{th-christoffel}, the $(m-1)\times (n-1)$ net satisfying condition~(ii$^*$) of Theorem~\ref{th-mxn} and only they are included into such pairs. By the metric duality, the result follows.   
\end{proof}

\section{Smooth flexibility}
\label{sec-smooth}



\subsection{Definition}


  Now we proceed to smooth flexible nets in $I^3$ and recall their definition, essentially given in~\cite{isometric-isotropic}. We focus on finite flexibility; the infinitesimal one is discussed in~\cite{isometric-isotropic}.


Fix a rectangle \( U = [\alpha, \beta] \times [\gamma, \delta] \subset \mathbb{R}^2 \).
A \emph{conjugate net} is a smooth function \( f\colon U\to I^3 \) 
such that for each $(u,v)\in U$ the derivatives $f_u(u,v)$ and $f_v(u,v)$ span a plane and $f_{uv}(u,v)$ is parallel to the plane. 
A conjugate net is \emph{admissible} if the plane spanned by $f_u(u,v)$ and $f_v(u,v)$ is non-isotropic. A \emph{
family of admissible conjugate nets} is a continuous 
function \( f\colon U\times [0,1]\to I^3 \) such that \( f(u,v, t) \) is an admissible conjugate net for each $t\in [0,1]$.

\fixskip
\begin{definition} \label{def-isom-surf} (Cf.~\cite[Definition~20]{isometric-isotropic}.)
An \emph{isotropic isometric deformation} of an admissible conjugate net \( f(u,v) \) 
is a 
family of admissible conjugate nets 
\( f(u,v, t) \) 
with \( f(u,v,0) = f(u,v) \) 
such that for each $(u,v,t)\in U\times [0,1]$ the two points $f(u,v,t)$ and $f(u,v)$
have the same top view and 
    the same isotropic Gaussian curvature.

An isotropic isometric deformation is \emph{trivial} if for each $t\in [0,1]$ there is an isotropic congruence $C_t$ such that $f(u,v, t)= C_t(f(u,v))$ for all $(u,v)\in U$. 
An admissible conjugate net is (\emph{one-sided}) \emph{flexible in $I^3$} if it has a nontrivial isotropic isometric deformation. 
\end{definition}

\subsection{Statement of the result}


Analogously to the discrete case, the classification of smooth flexible nets in $I^3$ reduces to the classification of smooth deformable nets. Note that the latter is more challenging than the discrete one; see Conjecture~\ref{conj-smooth-deformable} below.

We now give the relevant notions. Two conjugate nets $f,f^+\colon U\to \mathbb R^3$ are \emph{parallel} or \emph{Combescure transforms} of each other if $f_u(u,v)\parallel f_u^+(u,v)$ and $f_v(u,v)\parallel f_v^+(u,v)$ for each $(u,v)\in U$. 
A Combescure transform $f^+$ of a conjugate net $f$ is \emph{area-preserving}, if determinants of the first fundamental forms agree,
   $$(f_u\cdot f_u)(f_v\cdot f_v)-( f_u\cdot f_v)^2=(f^+_u\cdot f^+_u)(f^+_v\cdot f^+_v)-( f^+_u\cdot f^+_v)^2,$$ 
%
at each point $(u,v)\in U$, where $a \cdot b$ denotes the Euclidean scalar product of vectors $a$ and $b$.
Two conjugate nets $f,f^+\colon U\to \mathbb R^3$ are \emph{congruent}, if $f^+=C\circ f$ for some (isotropic or Euclidean) congruence $C\colon \mathbb R^3\to \mathbb R^3$.

\fixskip
\begin{definition} \label{def-smooth-deformable-net} \cite{pirahmad2024area}
A conjugate net $f(u,v)$ is called \emph{deformable} if it belongs to a continuous family of pairwise non-congruent (in Euclidean or isotropic sense) area-preserving Combescure transforms $f^+(u,v,t)$, where $t\in [0,1]$.
\end{definition}
\fixskip


For an admissible conjugate net $f\colon U\to I^3$, its \emph{metric dual} is the function $f^*\colon U\to I^3$ that takes a point $(u,v)\in U$ to the point $f^*(u,v)$ metric dual to the tangent plane at $f(u,v)$ (i.e., the plane passing through $f(u,v)$ parallel to $f_u(u,v)$ and $f_v(u,v)$).
The metric dual $f^*$ is also an admissible conjugate net if the isotropic Gaussian curvature of $f$ vanishes nowhere. (This follows from the equivalent definitions of conjugate directions in \cite[Section~60]{K91}.)
We have the following analog of Propositions~\ref{prop-infinitesimal-characterization} and~\ref{l-flexible-vs-deformable}; cf.~\cite[Proposition~21]{isometric-isotropic}.


\fixskip
\begin{prop} \label{l-smooth-flexible-vs-deformable}
An admissible conjugate net with nowhere vanishing isotropic Gaussian curvature is flexible in $I^3$ if and only if the metric dual net is deformable.
\end{prop}

\fixskip

Note that 
vanishing isotropic Gaussian curvature (at \emph{each} point of a non-planar surface) 
implies developability in Euclidean geometry and flexibility in $I^3$:
a nontrivial isotropic isometric deformation is then obtained by 
the composition with the transformation
$(x,y,z)\mapsto (x,y,(1-t)z)$.


\subsection{A class of flexible nets}
\label{ssec-classification-smooth-deformable}

Let us introduce a class of smooth nets flexible in $I^3$. 

The desired nets are metric dual to the so-called scale-translational surfaces. 
Recall that a \emph{smooth cone-cylinder net}, or \emph{scale-translational surface} with 
base curves $a(u)$ and $b(v)$ and scaling function $\sigma(u)$, is
\begin{equation}
    f^*(u,v)=a(u) + \sigma(u) b(v), \quad  (u,v)\in U, 
    \label{eq:cone-cylinder}
\end{equation} 
for some smooth functions $a\colon [\alpha,\beta]\to \mathbb{R}^3$, $b\colon [\gamma,\delta]\to \mathbb{R}^3$, $\sigma \colon [\alpha,\beta]\to \mathbb{R}$ such that 
$f^*_u(u,v)\nparallel f^*_v(u,v)$ for each $(u,v)\in U$, i.e., 
$ a'(u)+\sigma'(u)b(v)\nparallel \sigma(u)b'(v)$. In particular, $\sigma(u) \ne 0$ everywhere. Assume without loss of generality that $\sigma(u)>0$, otherwise change the sign of both $\sigma$ and $b$.
We see that $f^*(u,v)$ is a conjugate net; here $f^*_{uv}=\sigma'(u)b'(v)$ is even parallel to $f^*_v=\sigma(u)b'(v)$. Smooth cone-cylinder nets are a particular case of double cone-nets studied in \cite{kilian2023smooth}.

\fixskip
\begin{prop} \textup{(See \cite[Proposition~14]{pirahmad2024area})}
    \label{th-deform-surf-i} 
    A 
    conjugate net 
    has form~\eqref{eq:cone-cylinder} with $\sigma(u)>0$ if and only if 
    the tangents to the $u$-parameter lines at points of each $v$-parameter line are concurrent or parallel 
    and the tangents to the $v$-parameter lines at points of each $u$-parameter line are parallel. 
\end{prop}
\fixskip

\begin{prop} \textup{(See \cite[Proposition~15]{pirahmad2024area})}
    \label{th-deform-surf}
    Any regular cone-cylinder net~
    \eqref{eq:cone-cylinder}
    is deformable. For $\sigma(u)>0$, it is embedded into a one-parameter family
\begin{equation} \label{eq:cont-family}
\hspace*{-0.19cm}f^*(u,v,t):=
    a(\alpha)+\int_\alpha^u \frac{a'(w)\sigma(w)\,dw}{\sqrt{t+\sigma(w)^2}}
    +
    \sqrt{t+\sigma(u)^2}\,
    b(v),\ \ 
    t\in [0,1],\ 
    (u,v)\in U,
\end{equation} 
    of cone-cylinder nets which are related to each other by
    area-preserving Combescure transformations.
\end{prop}
\fixskip


Analogously to the discrete case, we introduce the following class of 
nets: 
\emph{admissible conjugate nets with planar parameter lines, where one family of parameter lines lie in isotropic planes, and the Gaussian curvature vanishes nowhere}. We call them \emph{generalized T-surfaces} or \emph{generalized smooth T-nets}. Their properties are parallel to the ones of generalized T-nets discussed in Section~\ref{ssec:Finite:Geometric properties}: Generalized T-surfaces
are metric dual to smooth cone-cylinder nets. 
Generalized T-surfaces can be viewed as the isotropic counterparts and a generalization of T-surfaces, a known class of smooth flexible nets in Euclidean geometry~\cite[Definition~4.1]{izmestiev2024isometric}. As we shall see now, all generalized T-surfaces are flexible in $I^3$. 
As usual, denote $e_1=(1,0,0)^\mathrm{T}$, $e_2=(0,1,0)^\mathrm{T}$, $ e_3=(0,0,1)^\mathrm{T}$.

\fixskip
\begin{prop}
    \label{p-generalized-T-surface}
    Any generalized smooth \( T \)-net is flexible in \( I^3 \). Such a net has the form  
    \begin{equation}
        \label{T-net}
        f(u,v) = -\frac{1}{\det(e_3, a'(u) + \sigma'(u)b(v), b'(v))}
        \small\begin{pmatrix}
            \det(e_1, a'(u) + \sigma'(u)b(v), b'(v)) \\
            \det(e_2, a'(u) + \sigma'(u)b(v), b'(v)) \\
            \det(a(u) + \sigma(u)b(v), a'(u) + \sigma'(u)b(v), b'(v))
        \end{pmatrix},
    \end{equation}
    where \( (u, v) \in U\) 
    and 
    \( a \colon [\alpha, \beta] \to \mathbb{R}^3 \), \( b \colon [\gamma, \delta] \to \mathbb{R}^3 \), \( \sigma \colon [\alpha, \beta] \to (0, +\infty) \) are smooth functions such that $e_3,b''(v),a''(u)+\sigma''(u)b(v)\not\in span(a'(u)+\sigma'(u)b(v),b'(v))$. 
    
    It has a nontrivial isotropic isometric deformation 
    \small\begin{equation}
        \label{iso-deform}
        f(u,v,t) = -\frac{1}{\det(e_3, a'(u)+\sigma'(u)b(v), b'(v))} \begin{pmatrix}
        \det(e_1, a'(u)+\sigma'(u)b(v), b'(v)) \\
        \det(e_2, a'(u)+\sigma'(u)b(v), b'(v)) \\
        \det(f^*(u,v,t), a'(u)+\sigma'(u)b(v), b'(v))
        \end{pmatrix},
    \end{equation}
    where $f^*(u,v,t)$
    is given by~\eqref{eq:cont-family}.
\end{prop}
\fixskip

The same expression can be compactly written in the homogeneous coordinates in terms of the \emph{triple cross product} in \( \mathbb{R}^4 \): 
$$
F(u,v)=(F_1(u,v),\dots, F_4(u,v))=F^*(u,v)\times F^*_u(u,v)\times F^*_v(u,v),
$$
where $F^*(u,v):=(-f^*_1, -f^*_2, 1, f^*_3)$ are the homogeneous coordinates of the plane, metric dual to a point \( (f^*_1, f^*_2, f^*_3):= a(u) + \sigma(u)b(v)  \) of the metric dual net, so that 
$$f(u,v)=\left(F_1(u,v),F_2(u,v),F_3(u,v)\right)/F_4(u,v).$$

\subsection{Proof of the results}

Proposition~\ref{l-smooth-flexible-vs-deformable} follows 
from the following stronger assertion.

\fixskip
\begin{lemma}\label{l-isomtric-deformation-vs-deformation}
   A family of 
   nets \( f(u, v, t) \) is a nontrivial isotropic isometric deformation of an admissible conjugate net 
   $f(u, v)$ with nowhere vanishing isotropic Gaussian curvature if and only if 
   the metric duals \( f^*(u, v, t) \) 
   are area-preserving Combescure transforms of the metric dual~$f^*(u, v)=f^*(u, v, 0)$, not all being congruent. 
\end{lemma}

\begin{proof}[Proof of Lemma~\ref{l-isomtric-deformation-vs-deformation}]
   The metric dual \( f^*(u, v) \) is also an admissible conjugate net with nowhere 
   vanishing isotropic Gaussian curvature (by 
   the equivalent definitions of conjugate directions discussed in \cite[Section~60]{K91} and the relation between the isotropic Gaussian curvatures of metric dual nets discussed in~\cite[Lemma~14]{isometric-isotropic} and~\cite{strubecker-dual}.
   ) Both isotropic isometric deformations and area-preserving Combescure transforms preserve the isotropic Gaussian curvature \cite[Remark~6]{pirahmad2024area}, thus we may assume that 
   both $f(u,v,t)$ and \( f^*(u, v, t) \) also have nonvanishing isotropic Gaussian curvature. 
   Then one of the nets $f(u,v,t)$ and \( f^*(u, v, t) \)
   is an admissible conjugate net if and only if the other one is.
   In what follows assume that both are.

   Consider the two conditions on the points $f(u,v,t)$ and $f(u,v)$ in Definition~\ref{def-isom-surf}.
   
   Two points $f(u,v,t)$ and $f(u,v)$ are parallel for each $(u,v,t)\in U\times [0,1]$, if and only if the tangent planes at \( f^*(u, v, t) \) and $f^*(u, v)$ are parallel. The latter is equivalent 
   to the parallelism of the conjugate nets \( f^*(u, v, t) \) and $f^*(u, v)$, because their isotropic Gaussian curvature vanishes nowhere.
   
   Further, since the net $f(u,v,t)$ has nowhere vanishing isotropic Gaussian curvature \( K(u,v,t) \), it has a locally injective isotropic Gauss map. As noted in the introduction, then the total isotropic Gaussian curvature \( \int_N K(u,v,t) \, dudv \) of a small enough neighborhood $N\subset U$ equals the oriented isotropic area of the dual surface $f^*(N\times \{t\})$. Thus the isotropic Gaussian curvature at two parallel points $f(u,v,t)$ and $f(u,v)$ is equal if and only if \( f^*(u,v,t) \) is an \emph{isotropic}-area-preserving Combescure transform of \( f^*(u,v) \). Since Combescure transforms have parallel tangent planes at corresponding points, this is equivalent to being area-preserving. 

   Finally, if for some $t$, the admissible nets $f(u,v,t)$ and $f(u,v)$ are related by an isotropic congruence, then the latter preserves the top view and thus has form $(x,y,z)\mapsto \left(x,y,z+c_1x+c_2y-c_3\right)$ for some $c_1,c_2,c_3\in \mathbb{R}$. Then $f^*(u,v,t)$ and $f^*(u,v)$ are related by the translation $(x,y,z)\mapsto \left(x+c_1,y+c_2,z+c_3\right)$. Conversely, if $f^*(u,v,t)$ and $f^*(u,v)$ are related by an isotropic (or Euclidean) congruence $C$, then $f(u,v,t)$ and $f(u,v)$ are, because $C$ can only be a translation.
\end{proof}

\begin{proof}[Proof of Proposition~\ref{p-generalized-T-surface}]
Consider a generalized smooth T-net $f(u,v)$ and its metric dual $f^*(u,v)$. The metric duality takes a parameter line $u=\mathrm{const}$ of $f(u,v)$ to the set of tangent planes along a parameter line $u=\mathrm{const}$ of $f^*(u,v)$.
If the former parameter line is planar, then the tangent planes have a common point (an improper one if the parameter line lies in an isotropic plane). By the equivalent definition of conjugate directions in \cite[Section~60]{K91}, the tangents to the parameter lines $v=\mathrm{const}$ at the points of the parameter line $u=\mathrm{const}$ of $f^*(u,v)$ pass through this common point. By Propositions~\ref{th-deform-surf-i} and~\ref{th-deform-surf}, the net $f^*(u,v)$ has form~\eqref{eq:cone-cylinder} and is deformable.  By Proposition~\ref{l-smooth-flexible-vs-deformable}, the net $f(u,v)$ is flexible in \( I^3 \).

Since $f(u,v)$ is an admissible conjugate net with nowhere vanishing isotropic Gaussian curvature, $f^*(u,v)$ is. Thus \( \det(e_3,f^*_u, f^*_v) \neq 0 \).
By~\cite[Equation~(5.1)]{strubecker3},
the isotropic Gaussian curvature of $f^*(u,v)$ is given by
\begin{equation}
    \label{form-is-gauss}
K=\frac{\det(f^*_u,f^*_v,f^*_{uu})\det(f^*_u,f^*_v,f^*_{vv})-\det(f^*_u,f^*_v,f^*_{uv})^2}{\left(\overline{{f}^*_u}\cdot \overline{{f}^*_u}\right)\left(\overline{{f}^*_v}\cdot \overline{{f}^*_v}\right)-\left(\overline{{f}^*_u}\cdot \overline{{f}^*_v}\right)^2}.
\end{equation}
Thus $\det(f^*_u, f^*_v, f^*_{uu})\det(f^*_u, f^*_v, f^*_{vv}) \neq 0.$
So, $f(u,v)$ is a generalized smooth T-net if and only if $f^*(u,v)$ has form~\eqref{eq:cone-cylinder} with $a'(u) + \sigma'(u)b(v)\nparallel b'(v)$ and
    \[
    e_3, b''(v), a''(u) + \sigma''(u)b(v) \notin \mathrm{span}(a'(u) + \sigma'(u)b(v), b'(v)).
    \]

By a straightforward computation, we have (see e.g.,~\cite[Equation~(53)]{Yorov-etal})
\begin{equation}
    \label{eq-det-dual}
    f(u,v) = -\frac{1}{\det(e_3,f^*_u,f^*_v)} 
        \begin{pmatrix}
            \det(e_1,f^*_u,f^*_v)\\
            \det(e_2,f^*_u,f^*_v)\\
            \det(f^*,f^*_u,f^*_v)
        \end{pmatrix}.  
\end{equation}
Substituting~\eqref{eq:cone-cylinder} and~\eqref{eq:cont-family} into~\eqref{eq-det-dual} and canceling \( \sigma(u)\ne 0 \), we obtain the desired formulae~\eqref{T-net} and~\eqref{iso-deform}. The latter is a nontrivial isotropic isometric deformation by Proposition~\ref{th-deform-surf} and Lemma~\ref{l-isomtric-deformation-vs-deformation}.
\end{proof}

\section{Conclusion and open problems} 
\label{sec:concl-open-pr}

In the previous sections, we classified flexible and infinitesimally flexible $m\times n$ nets in~$I^3$, and constructed examples of smooth flexible nets in $I^3$. The infinitesimally flexible nets are essentially the same in isotropic and Euclidean geometry. The classification of flexible $m\times n$ nets in $I^3$ is much simpler: we get just two classes with visual geometric descriptions, in sharp contrast to the Euclidean case. An example of a flexible net in $I^3$ from the first class  (``flower'') is shown in Figure~\ref{fig:flexible-eucl} to the top. 

Using the resulting isotropic flexible net to initialize a numerical optimization, we produce an approximate Euclidean flexible net 
in Figure~\ref{fig:flexible-eucl} to the bottom. Surprisingly, the optimization leads to a very small (although \emph{not} neglectable) shape change to achieve Euclidean flexibility to a very high precision (enough for any practical applications). So, the isotropic flexible net seems close to some Euclidean one. The discussion of the used optimization algorithm goes beyond the scope of the present paper and is going to be presented in a subsequent publication~\cite{jiang-2025} based on earlier work~\cite{quadmech-2024}. 



Let us list a few open questions.
First, it would be interesting to drop the convexity assumptions made throughout this work. For a general Q-net, the definition of flexibility in $I^3$ is literally the same, as long as the planes of the faces are not isotropic. We expect the classification theorem (Theorem~\ref{th-flexible}) to remain the same under a minor restatement of condition~(i). However, the proof will require additional steps because the convexity was essentially used in the previous work \cite{pirahmad2024area} we rely on. 

\fixskip
\begin{problem}[Non-dual-convex nets] Characterize all Q-nets  that are flexible in $I^3$.
\end{problem}
\fixskip

In particular, it is interesting to characterize isotropic origami. 

\fixskip
\begin{problem}[Isotropic Origami]
Characterize all planar nets  that are flexible in $I^3$.
\end{problem} 
\fixskip

Among nets with planar faces, the ones with faces inscribed in circles are of special interest. These \emph{circular nets} are discrete analogs of principal nets. \emph{Isotropic circular nets} are Q-nets with the faces inscribed in isotropic circles~\cite{pottmann2007discrete}.  \emph{Isotropic conical nets} are defined analogously to conical nets~\cite{pottmann2007discrete}.

\fixskip
\begin{problem}[Flexible isotropic circular nets] What flexible nets in $I^3$ are isotropic circular? conical?
\end{problem}
\fixskip


One can easily adopt the definition of flexibility in $I^3$
beyond square-grid combinatorics, in particular, introducing \emph{flexible polyhedra} in $I^3$; cf.~\cite{Gaifullin-25}. They are closed polyhedral surfaces whose metric duals admit continuous families of noncongruent area-preserving Combescure transformations. The latter is not possible for a polyhedron with a convex metric dual, by Minkowski's theorem on the existence and uniqueness of a convex polyhedron with given directions and areas of faces.
This implies \emph{isotropic 
Cauchy's rigidity 
theorem}: a polyhedron whose metric dual is convex is not flexible in~$I^3$ \cite[Section~1.1]{pirahmad2024area}. 

\fixskip
\begin{problem}[Flexible polyhedra] Construct examples of flexible polyhedra in $I^3$. Characterize all flexible octahedra in $I^3$.
\end{problem}
\fixskip

The famous bellows conjecture (Sabitov's theorem) asserts that the volume of a flexible polyhedron is preserved under a Euclidean isometric deformation. In other words, one cannot make bellows from a flexible polyhedron (with rigid faces). In Euclidean geometry, this was proved by I.~Sabitov in the 1990s. The same was proved in hyperbolic geometry by A.~Gaifullin~\cite{Gaifullin-15} and disproved in spherical geometry by V.~Alexandrov~\cite{Alexandrov-97, Gaifullin-25}.

\fixskip
\begin{problem}[Isotropic bellows conjecture]
    Is the volume of a flexible polyhedron in $I^3$ preserved under an isotropic isometric deformation?
\end{problem}
\fixskip

In this paper, we focused on \emph{one-sided flexibility}, meaning that the family parameter $t$ assumed values in $[0,1]$ with $t=0$ corresponding to the original net. See Definitions~\ref{def-flexibility} and~\ref{def-isom-surf}. \emph{Two-sided flexibility} is defined analogously, only the condition $t\in [0;1]$ is replaced with $t\in [-1;1]$ and we require all the nets in a \emph{nontrivial two-sided isometric deformation} to be pairwise non-congruent. This applies to smooth and discrete setups, as well as Euclidean and isotropic geometry. See~\cite{izmestiev2024isometric}.

\fixskip
\begin{problem}[One- vs two-sided flexibility] Prove that one-sided flexibility implies a two-sided one in isotropic geometry, and does not imply a two-sided one in Euclidean geometry.
\end{problem}
\fixskip

As observed in Section~\ref{ssec:Finite:Geometric properties}, any T-net, after an appropriate Euclidean rotation, is flexible in both isotropic and Euclidean geometries. 

\fixskip
\begin{problem}[Both isotropic and Euclidean flexibility]
    Are T-nets the only Q-nets that are flexible in both isotropic and Euclidean geometries?
\end{problem}
\fixskip

Yet another promising direction is the second-order infinitesimal flexibility, both in the smooth and discrete setup. In the Euclidean setting, Schief et al.~\cite{Schief2008} characterized it in terms of so-called Bianchi surfaces and their discrete counterparts.

\fixskip
\begin{problem}[Second-order isotropic flexibility]
What nets are second-order infinitesimally flexible in $I^3$? Is there a characterization similar to the Euclidean one? Are all second-order infinitesimally flexible smooth conjugate nets in $I^3$ finitely flexible in $I^3$?
\end{problem} 
\fixskip

Another interesting problem is the characterization of all deformable smooth nets, which is equivalent to characterizing all flexible smooth nets in \(I^3\) via the metric duality.
Recall that a \emph{reparametrization} of a smooth net $f(u,v)$ is a net $F(U,V)$ such that $F(U(u),V(v))=f(u,v)$ for a pair of smooth functions $U(u)$ and $V(v)$ with nowhere vanishing derivatives. Set $\int_{a}^b h(z)\,dz:=-\int _{b}^a h(z)\,dz$ for $b<a$.

\fixskip
\begin{conjecture}[Classification of smooth deformable nets] \label{conj-smooth-deformable}
    A net \( f(u,v) \) is a deformable conjugate net if and only if in a sufficiently small neighborhood of any generic point $(u_0,v_0)$,
    the net \( f(u,v) \) is scale-translational 
    or 
    has the following form up to translation and reparametrization:
    \begin{equation}
        \label{deform-net-smooth}
        f(u,v) = \sqrt{uv} \underbrace{\left[\int_0^1 \frac{\Phi\left(u + (v - u) z \right)}{\sqrt{z - z^2}} \,dz + \int_0^1 \frac{\Psi\left(u + (v - u) z \right)\log\left(z(1 - z)(v - u)\right)}{\sqrt{z - z^2}} \,dz\right]}_{h(u,v)},
    \end{equation}
    where \( \Phi \) and \( \Psi \) are arbitrary smooth 
    vector-functions, possibly depending on~$(u_0,v_0)$.

    Moreover, net~\eqref{deform-net-smooth} belongs to the following continuous family of pairwise non-congruent area-preserving Combescure transformations in a neighborhood of~$(u_0,v_0)$:
    $$f(u,v,t)=h(u,v)\sqrt{(t+u)(t+v)}-t\int_{u_0}^u \sqrt{\frac{t+v}{t+\bar u}}h_u(\bar u,v)\,d\bar u-t\int_{v_0}^v \sqrt{\frac{t+u_0}{t+\bar v}}h_v(u_0,\bar v)\,d\bar v.$$
\end{conjecture}
\fixskip


This conjecture is paradoxical: the variety of smooth deformable nets is much larger than the smooth limits of discrete deformable nets (those limits are just scale-translational nets). 

It may be possible to identify a deeper underlying reason for the existence of only two classes of flexible \(3 \times 3\) nets in \(I^3\), whereas the Euclidean counterpart, as classified by Izmestiev \cite{izmestiev2017classification}, includes a significantly larger variety.

The idea of initializing a numerical optimization algorithm
for the computation of a Euclidean structure by an isotropic
counterpart goes beyond the context of flexibility
and deserves to be investigated for other difficult problems as well. For instance, see \cite{MOROZOV2021, Yorov-etal,jiang-2025}. Further, one can try to apply this idea to non-Euclidean structures as well; cf.~\cite{Schlenker-02}.

\subsubsection*{Acknowledgements}

The authors are grateful to KAUST baseline funding for supporting this research work and to Caigui Jiang for the images in Figure~\ref{fig:flexible-eucl}. Those in the bottom row are based on his implementation of an optimization algorithm for the computation
of Euclidean flexible nets \cite{quadmech-2024}. 
We thank Alexander Gaifullin, Ivan Izmestiev, and Christian M\"uller for useful discussions.

\bibliography{sn-bibliography}

\notforarxiv

\section*{Statements \& Declarations}

\textbf{Funding.} This research has been supported by KAUST baseline funding. 
\vspace{0.2cm}\\
\textbf{Data availability.}  Data sharing is not applicable to this article as no datasets were generated or analyzed during the current study.
\vspace{0.2cm}\\
\textbf{Conflict of interest.} The authors have no relevant financial or non-financial interests to disclose.
\vspace{0.2cm}\\
\textbf{Author contributions.} All authors contributed equally to this work. 

\endnotforarxiv

\end{document}